\documentclass[12pt]{amsart}
\usepackage{amsaddr}
\usepackage{enumitem}

\newcommand{\Cdb}{\mbox{$\mathbb{C}$}}

\newcommand{\Ndb}{\mbox{$\mathbb{N}$}}

\newcommand{\Rdb}{\mbox{$\mathbb{R}$}}

\newcommand{\A}{\mbox{${\mathcal A}$}}

\newcommand{\E}{\mbox{${\mathcal E}$}}
\newcommand{\F}{\mbox{${\mathcal F}$}}

\renewcommand{\H}{\mbox{${\mathcal H}$}}

\newcommand{\K}{\mbox{${\mathcal K}$}}

\renewcommand{\S}{\mbox{${\mathcal S}$}}

\newcommand{\Tr}{\text{Tr}}
\newcommand{\La}{\Lambda}
\newcommand{\Ran}{\text{Ran}}
\newcommand{\simp}{\underset{\geq 0}{\sim}}

\newtheorem{theorem}{Theorem}[section]
\newtheorem{lemma}[theorem]{Lemma}

\newtheorem{proposition}[theorem]{Proposition}
\newtheorem{definition}[theorem]{Definition}
\theoremstyle{remark}
\newtheorem{remark}[theorem]{\bf Remark}
\theoremstyle{definition}

\numberwithin{equation}{section}

\begin{document}

\title[]{Positive contractive projections in Schatten Spaces}

\author{Estelle Boffy}
\address{Laboratoire de Math\'ematiques, Universit\'e Marie et Louis Pasteur
, 25030 Besan\c con Cedex, France}
\email{estelle.boffy@math.cnrs.fr}

\date{\today}

\begin{abstract} We characterize the positively 1-complemented subspaces of $S^p$, for $1\leq p<\infty$, where $S^p$ denotes the Schatten spaces. Building on the work of Arazy and Friedman, who described the 1-complemented subspaces of $S^p$, for $1\leq p\neq 2  <\infty$, we establish that there are five mutually distinct types of indecomposable positively 1-complemented subspaces in $S^p$. Moreover, every positively 1-complemented subspace of $S^p$ can be expressed as a direct sum of some of these indecomposable subspaces. The proof first uses the Arazy-Friedman result to identify which 1-complemented subspaces of $S^p$, for $1\leq p\neq 2<\infty$, are positively 1-complemented. Then, in a second step, it treats the special case $p=2$ by relating it to the previously studied case $p\neq 2$.
\end{abstract}

\maketitle

{\it Keywords: Projections, complemented subspaces, Schatten spaces, positive maps.

 2020 Mathematics subject classification: Primary 47B10, 46B28.}

\section{Introduction}\label{Intro}
Historically, the study of contractive projections on Banach spaces and the structure of their ranges have been a central theme in the geometry of these spaces. For a comprehensive account of the existing literature on this topic, see \cite{Randri}. The study on classical (i.e., commutative) $L^p$ spaces was conducted by Douglas (1965),  Ando (1966) and later by Bernau and Lacey (1974); see \cite{Abra}, \cite{Ando}, \cite{Bernau}, and \cite{Douglas}. In particular, Ando provided a complete description of all contractive projections $P:L^p(\Omega,\mu)\to L^p(\Omega,\mu)$, where $(\Omega,\mu)$ is a probability space, and $1\leq p\neq 2<\infty$. Note that when $p=2$, contractive projections coincide with orthogonal projections. However, Ando proved that if one further assumes that a contractive projection $P:L^2(\Omega,\mu)\to L^2(\Omega,\mu)$ is \textit{positive}, then the description obtained for $p\neq 2$ remains valid.

In the setting of Schatten spaces, for $H$ and $K$ some separable Hilbert spaces, an explicit description of the structure of 1-complemented subspaces of $S^p(H,K)$ for $p\neq 2$ is given by Arazy and Friedman in 1978 and 1992 in \cite{AF2}, and \cite{AF3}. Recall that we denote by $S^p(H,K)$ (or simply $S^p$) the Schatten space of all bounded operator $x : H\to K$ such that $\Tr(|x|^p)<\infty$. A closed subspace $X$ of the space $S^p(H,K)$ is said to be 1-complemented if it is the range of a contractive projection $P: S^p(H,K)\to S^p(H,K)$. According to \cite[Proposition 1.2, p.8]{AF3}, for $1<p<\infty$, this projection is necessarily unique. Therefore, the characterization of contractive projections is equivalent to that of 1-complemented subspaces in this case. Building on the Arazy and Friedman description, Le Merdy, Ricard and Roydor proved in \cite{MRR} (2009) that, for $1\leq p\neq 2<\infty$, the completely 1-complemented subspaces of $S^p$ are exactly the subspaces that are completely isometric to an $\ell^p-$direct sum of subspaces of the form $S^p(H,K)$.

On general noncommutative $L^p$ spaces, defined from a von Neumann algebra equipped with a normal semifinite faithful (nsf) trace, there is no description of contractive projections. Nevertheless, Arhancet and Raynaud, in \cite{AR} (2024), established that the range of a 2-positive contractive projection on a noncommutative $L^p$ space is completely order and completely isometrically isomorphic to a noncommutative $L^p$ space. Their approach is different from the methods of Arazy and Friedman. The structure of positively 1-complemented subspaces of noncommutative $L^p$ spaces is investigated in \cite{A} by Arhancet. It is shown that, in a wide range of cases, the range of a positive contractive projection is isometric to a nonassociative $L^p$ space associated with a $JW^*$-algebra. The same author studied the range of contractively decomposable projections on noncommutative $L^p$ spaces in \cite{A_dec} (2024).

The goal of the present paper is to obtain an explicit description of all positively 1-complemented subspaces of $S^p(H):=S^p(H,H)$, that is, the subspaces of $S^p(H)$ that are the range of a contractive and positive projection $P : S^p(H)\to S^p(H)$. Although the Schatten spaces $S^p(H)$ form a particular class of noncommutative $L^p$ spaces, this description does not follow from \cite{A} and is specific to the space $S^p(H)$.

 All Hilbert spaces considered in this article are assumed to be separable. The main theorem of this paper is Theorem \ref{THM1}, which describes the positively 1-complemented subspaces of $S^p(H)$, for $1\leq p<\infty$. The proof proceeds in two steps. We first prove the theorem in the case $p\neq 2$, relying on the description of 1-complemented subspaces of $S^p(H)$ due to Arazy and Friedman, and adding the assumption that the projection is positive. For the case $p=2$, the idea is to reduce to the previous step. 

\subsection*{Structure of the paper}

\begin{itemize}
\item \textbf{Section \ref{sectionDescr 1compl}} gives a brief overview of the essential notions and notations required for the rest of the paper. The description of 1-complemented subspaces of $S^p$, for $p\neq 2$, is reviewed in Theorem \ref{thmAF}. We use the work of Le Merdy, Ricard, and Roydor \cite{MRR} (see Theorem 2.8 and Section 5), which clarifies the description originally due to Arazy and Friedman. A relevant notion of equivalent subspace introduced in \cite{MRR} is deepened. 

\item \textbf{Section 3} is devoted to $JC^*$-triples. We review basic properties of $JC^*$-triples, relying on \cite{H2} and on \cite{H}; and highlight important properties that will be used in the rest of the paper. 

\item \textbf{Section 4} begins with a review of the results on the support projection of a positive contractive projection that appear in \cite{AR}. Lemma \ref{decpos1compl} allows us to restrict the study of positively 1-complemented subspaces to the study of \textit{indecomposable} positively 1-complemented subspaces. Then, we state the main result of the paper: Theorem \ref{THM1}, which presents a classification of \textit{positively} 1-complemented subspaces of $S^p(H)$. 

\item\textbf{Section 5} gives the proof of the case $p\neq 2$ of Theorem \ref{THM1}, relying on the classification of contractive projections by Arazy and Friedman, and determining, for each type, the conditions to have a positively 1-complemented subspace. The case $p=1$ is discussed at the end of the section.

\item \textbf{Section 6} provides the proof of Theorem \ref{THM1} in the case $p=2$, based on the case $p\neq 2$ of the same theorem. Note that the case of $S^2(H)$ is particular. In this space, the contractive projections are the orthogonal projections, but one may consider the case of positively 1-complemented subspaces of $S^2(H)$, even though the Arazy-Friedman theorem no longer applies. To a positively 1-complemented subspace $X$ of $S^2(H)$, we associate some other positively 1-complemented subspace $X_q$ of $S^q(H)$, for $q\neq 2$. We use the result established in the case $q\neq 2$ to describe $X_q$, and we exploit the link between the subspaces $X_q$ and $X$ to obtain a description of the original subspace $X$.

\item \textbf{Section 7} is an appendix that proves the equivalence between certain 1-complemented subspaces appearing in the classification by Arazy and Friedman, and subspaces appearing in the framework of Le Merdy, Ricard, and Roydor.
\end{itemize}

\section{Description of 1-complemented subspaces and background}\label{sectionDescr 1compl}

To understand the description of 1-complemented subspaces, we need to define the notations and definitions used throughout the paper. 
We assume that the reader is familiar with Schatten spaces $S^p(H,K)$, for which we refer e.g. to \cite{McCarthy}.

We use the natural embedding for tensor products. Let $H_1,$ $K_1$, $H_2$ and $K_2$ be Hilbert spaces. Then for $1\leq p<\infty$, we have \[S^p(H_1,K_1)\otimes S^p(H_2,K_2) \subset S^p\left(H_1\overset{2}{\otimes}H_2,  K_1\overset{2}{\otimes}K_2\right)\] where $\overset{2}{\otimes}$ denotes the Hilbertian tensor product. The left-hand side of this inclusion is a dense subspace of the right-hand side. For any $a\in S^p(H_1,K_1),~b\in S^p(H_2,K_2)$, $\|a\otimes b\|_p = \|a\|_p\|b\|_p$. Note that if one of the spaces $S^p(H_i,K_i)$, $i=1,2$ is finite-dimensional, then the inclusion becomes an equality. Also observe that if  $a\in S^p(H)$ and $b\in S^p(K)$  are two positive operators, then $a\otimes b$ is positive. Finally, the space $\ell^2_I\overset{2}{\otimes}H_1$ can be naturally identified with $\ell^2_I(H_1)$.

For $1\leq p<\infty$, and a sequence of Banach spaces $(X_i)_i$, we define the $p-$direct sum \[\bigoplus^p_i X_i := \{(x_i)_i,\quad x_i\in X_i,~\sum_i \|x_i\|^p<\infty\},\] endowed with the norm \[\|(x_i)_i\| := \left(\sum_i \|x_i\|^p\right)^{\frac{1}{p}},\] it is a Banach space. If $p=\infty$, we set \(\bigoplus\limits^\infty_i X_i=\{(x_i)_i,\quad x_i\in X_i,~(\|x_i\|)_i \text{ bounded}\}\), endowed with the norm \[\|(x_i)_i\| := \text{sup}_i\|x_i\|.\]
For $(H_i)_i$, and $(K_i)_i$ families of Hilbert spaces, we have the embedding \[ \bigoplus^p_i S^p(H_i,K_i) \subset S^p\left(\bigoplus^2_i H_i, \bigoplus^2_i K_i\right),\] defining an element $x = (x_i)_i\in\bigoplus\limits^p_i S^p(H_i,K_i) $ as \[x((h_i)_i) := (x_i(h_i))_i,\quad \text{for }(h_i)_i\in \bigoplus\limits^2_i H_i.\] Observe that this embedding determines a positive cone on \(\bigoplus\limits^p_i S^p(H_i,K_i)\). We note that an element $(x_i)_i \in \bigoplus\limits^p_i S^p(H_i,K_i)$ is positive if and only if each $x_i$ is positive.

Throughout the article, we use $1$ or $\operatorname{Id}_X$ to denote the identity operator on a space $X$. When using the notation $1$, the space on which the identity operator acts will be clear from the context.

In section 4, results about duality of Schatten spaces will be used. For $1\leq p< \infty$, and $q$ such that $\dfrac{1}{p} + \dfrac{1}{q} = 1$, we have an isometric isomorphism $S^q(H) \simeq S^p(H)^*$ through the map \[x\in S^q(H)\mapsto \left[y\in S^p(H)\mapsto \Tr(xy)\right],\] and we use the notation \[\langle x , y\rangle:= \Tr(xy),\quad x\in S^q(H),~y\in S^p(H)\] to denote the duality brackets. Note that we use the convention $S^\infty(H,K):=B(H,K)$. When there is a possible ambiguity about the Hilbert, we will denote by $\Tr_H$ the usual trace on $S^1(H).$ 

For the following lemma, note that using the identification \(S^p(\ell^2_n)\otimes S^p(H) = S^p(\ell^2_n(H)),\) we have a positive cone on the tensor product $S^p(\ell^2_n)\otimes S^p(H)$. We use the notation $S^p_n := S^p(\ell^2_n)$.
\begin{lemma}\label{linear_form_cp}
Let $H$ be a Hilbert space, let $1\leq p< \infty$, and let $\varphi : S^p(H)\to \Cdb$ be a positive linear form. Then for every $n\in\Ndb$, \begin{enumerate}
    \item  The map $\operatorname{Id}_{S^p_n}\otimes \varphi$ is positive.
    \item $\|\operatorname{Id}_{S^p_n}\otimes \varphi : S^p(\ell^2_n)\otimes S^p(H)\to S^p(\ell^2_n)\|\leq \|\varphi\|.$
\end{enumerate}
\end{lemma}
\begin{proof}
This is a particular case of \cite[Proposition 2.3, p.19 and Theorem 2.19, p.18]{Arh_Krieg}.

\end{proof}
\begin{remark}\label{Totimesphi}
\begin{enumerate}
\item The previous lemma can also be proved by elementary means.
    \item In the setting of the previous lemma, the first point implies that for any Hilbert space $K$, the map $\operatorname{Id}_{S^p(K)}\otimes\varphi$ extends to a bounded map from $S^p\left(K\overset{2}{\otimes}H\right)$ to $S^p(K)$. We again denote by $\operatorname{Id}_{S^p(K)}\otimes\varphi$ its extension. This means that $\varphi$ is completely bounded (see \cite{Pisier} for more details). 
    The second point implies that the extended map $\operatorname{Id}_{S^p(K)}\otimes\varphi$ is also positive. This means that $\varphi$ is completely positive.
    \item The previous lemma implies that if $\varphi : S^p(H)\to \Cdb$ is a contractive and positive linear form, and $T : S^p(K)\to S^p(K)$ is a contractive and positive map, then the operator $T\otimes \varphi : S^p(K)\otimes S^p(H)\to S^p(K)$ extends to a positive contractive map from $S^p(H\overset{2}{\otimes}K)$ to $S^p(K)$. This follows from the identity \(T\otimes\varphi = T\circ (\operatorname{Id}_{S^p(K)}\otimes\varphi).\) 
\end{enumerate}

\end{remark}

In the following, we give the description of 1-complemented subspaces; see Arazy-Friedman \cite{AF3}, and Le Merdy, Ricard and Roydor \cite{MRR}. The following is a brief summary of \cite[Sections 2 and 5]{MRR}.
To state the description, we use the following equivalence relation, introduced in \cite{MRR}.
Let $H,H',K,K'$ be Hilbert spaces, and $1\leq p < \infty$.
We say that two closed subspaces $X\subset S^p(H,K)$ and $Y\subset S^p(H',K')$ are equivalent, denoted \[X\sim Y,\]
if there exist partial isometries $U:H'\to H$ and $V : K'\to K$ such that

\begin{equation}\label{defequiv}
    X = VYU^* \mbox{  and   } Y=V^*XU.
\end{equation}
We prove that it is an equivalence relation in Remark \ref{rel_equiv}.
It is straightforward to verify that if $X\subset S^p(H,K)$ and $Y\subset S^p(H',K')$ are equivalent, and $P : S^p(H,K)\to S^p(H,K)$ is a contractive projection with range equal to $X$, then the map \begin{equation}\label{proj_equiv_spaces}
    Q : S^p(H', K') \to S^p(H', K'), \quad Q(y) = V^* P(VyU^*) U
\end{equation} is a contractive projection with range equal to $Y$.

Let $x,y\in S^p(H,K)$. We say that $x$ and $y$ are \textit{disjoint} operators if \[x^*y = 0 \quad \mbox{ and } \quad xy^* = 0.\]
This terminology differs from the one used in \cite{AF2} and \cite{MRR} where such elements are called orthogonal. It has been changed here to avoid confusion when working on $S^2(H)$. Note that if $x,y\in S^p(H)$ are disjoint, then \begin{equation}\label{relation_norm_disjoint}
    \|x+y\|_p^p = \|x\|^p_p + \|y\|^p_p.
\end{equation}

Two subspaces $X_1,X_2$ of the space $S^p(H,K)$ are called \textit{operator-disjoint} if every pair $(x_1,x_2)\in X_1\times X_2$, consists of disjoint  operators. In this case, the space $X_1\oplus X_2 \subset S^p(H,K)$ coincide with the space $X_1\overset{p}{\oplus}X_2$. Here, we shall not distinguish between the internal and external direct sum. Note that for $1\leq p\neq 2<\infty$, if the identity (\ref{relation_norm_disjoint}) holds for every pair $(x,y)\in X_1\times X_2$, then the subspaces $X_1$ and $X_2$ are operator-disjoint. This follows the equality case in the Clarkson inequality (see \cite[Theorem 2.7, p.261]{McCarthy}).

A subspace $X$ of the space $S^p(H,K)$ is said to be \textit{indecomposable} if it cannot be written as the $p$-sum of two nontrivial operator-disjoint subspaces.

For \( 1 < p < \infty \), it follows from \cite{AF3} that every subspace \( X\) of the space \(S^p(H, K) \) admits a direct sum decomposition: \[X = \bigoplus^p_{\alpha} X_\alpha,\] where each \( X_\alpha \) is indecomposable and the subspaces \( X_\alpha \) are pairwise operator-disjoint. Note that such a decomposition is unique up to the ordering of the family $(X_\alpha)_\alpha$. Moreover, $X$ is 1-complemented if and only if every subspace $X_\alpha$ is 1-complemented.

\medskip

In their memoir \cite{AF3}, Arazy and Friedman give a description of indecomposable and 1-complemented subspaces of \( S^p(H, K) \) for \( 1 < p < \infty \). This gives rise to six distinct types of such subspaces, which we will define below.
The first three types are relatively straightforward to define.

\medskip

Let \( I \) and \( J \) be two index sets. We denote by $S^p_{I,J}$ the space $S^p(\ell^2_J,\ell^2_I)$, and by $S^p_I$ the space $S^p(\ell^2_I,\ell^2_I)$. Then, let $\top : S^p_{I,J} \to S^p_{J,I}$ denote the transpose map.
Regarding elements of $S^p_{I,J}$ as scalar matrices, \[([t_{ij}]_{i\in I, j\in J})^\top = [t_{ji}]_{j\in J,i\in I}.\]

Next, we define the spaces of symmetric and anti-symmetric matrices within \( S^p_I \), respectively, as follows \[\S^p_I := \{w\in S^p_I, \quad w^\top = w\} \quad \mbox{ and } \quad \A^p_I := \{w\in S^p_I, \quad w^\top = -w\}.\]

\begin{definition}\label{type123}

A space of symmetric matrices (resp. of anti-symmetric matrices) is a space equivalent to $\S^p_I\otimes a$, with $|I|\geq 1$ (resp. $\A^p_I\otimes a$, with $|I|\geq 5$), for some countable index set $I$, and some operator $a\in S^p(H_1).$

A space of rectangular matrices is a space equivalent to \[\{(w\otimes a_1, w^\top\otimes a_2),~w\in S^p_{I,J}\} \subset S^p\left(\ell^2_J(H_1),\ell^2_I(H_1)\right)\overset{p}{\oplus}S^p\left(\ell^2_I(H_2),\ell^2_J(H_2)\right),\]
 for some countable index sets $I$ and $J$, with $|I|, |J|\geq 2$, or $|I|=1$ and $|J|=\infty$, and for some operators $a_i \in S^p(H_i)$, $i=1,2.$

\end{definition}

To define the fourth type, we first introduce some terminology and notation related to Fock spaces. For more details, we refer to \cite{MRR}.

\medskip 
Let $N\geq 2$ be an integer, and let $\mathcal{H}=\ell^2_N$, with $(e_k)_{1\leq k\leq N}$ denoting the canonical basis. For any $m=0,\ldots ,N$, let $\H^{\wedge m}:= \H\wedge\cdots \wedge \H$ denote the $m$-fold anti-symmetric tensor product of $\H$. We define \( \Lambda_N \) as the Hilbertian direct sum of these tensor products:
\[
\Lambda_N := \bigoplus_{0\leq m\leq N}^{2} \H^{\wedge m}.
\]
This space is known as the Fock space associated with \( \H \). Note that for $0\leq m\leq N$, \[\text{dim}(\H^{\wedge m}) = \binom{N}{m},\quad \text{ and }\quad \text{dim}(\La_N) = 2^N.\]

For $1\leq k\leq N$, the creation operator $c_k$ is defined by
 \[ c_k: \left|\begin{array}{ccc}
    \La_N & \to & \La_N \\
    x & \mapsto & e_k\wedge x
\end{array}\right. .\]

We need to consider restrictions of these operators.  Denote by \(x_{k,m} : \H^{\wedge(m-1)}  \to  \H^{\wedge m}\) the restriction of the operator $c_k$ to the space $\H^{\wedge(m-1)}$, for $1\leq k,m\leq N$. Then, for each \( 1 \leq m \leq N \), we define the map \begin{equation}\label{defvarphi_m}
    \varphi_m : \left|\begin{array}{ccc}
    \H & \to & S^p(\H^{\wedge (m-1)}, \H^{\wedge m})\subset S^p(\La_N) \\
    (t_1,\ldots ,t_N) & \mapsto  & \frac{1}{\binom{N-1}{m-1}^{1/p}}\sum\limits_{k=1}^N t_k x_{k,m}.
\end{array} \right.
\end{equation} According to \cite[Proposition 2.4, p.18]{AF2} which is extended for $1<p\neq 2<\infty$ in \cite{AF3}, it is an isometry.

\begin{definition}\label{type4}
An $N$-dimensional AF-Hilbertian space of finite dimension $N\geq 2$ is a space equivalent to \[\{(\varphi_1(t)\otimes a_1,\ldots ,\varphi_N(t)\otimes a_N),~t\in \H\} \subset \bigoplus^{p}_{1\leq m\leq N} S^p(\Lambda_N\overset{2}{\otimes} H_m),\] for $H_1$,\ldots ,$H_N$ some Hilbert spaces and some operators \( a_i \in S^p(H_m) \), \( 1 \leq m \leq N \).
\end{definition}

\medskip


\begin{definition} Let \( N \in \mathbb N^* \), $H$ a Hilbert space and let \((w_1, \ldots, w_N)\) be a family of operators in $B(H)$. This family is a spin system in $B(H)$ if \begin{itemize}
    \item for any \(j\in \{1,\ldots,N\},~w_j\) is a self-adjoint unitary operator,
    \item for any \(1\leq i\neq j\leq N\), we have $w_iw_j + w_jw_i = 0.$
\end{itemize}
\end{definition}

We now define a particular spin system. Let \( N \in \mathbb N^* \), and for each \( 1 \leq k \leq N \), recall \( c_k \in B(\Lambda_N) \). We define
\begin{equation}\label{defs_j}
    s_k := c_k + c_k^* \quad \text{and} \quad s_{-k} := \frac{c_k - c_k^*}{i}.
\end{equation}
It is well known that the family \((c_k)_{1\leq k\leq N}\) satisfies the relations \begin{equation}\label{CAR}
    \left\{\begin{array}{cc}
        c_ic_j + c_jc_i = 0, &  1\leq i,j\leq N  \\
        c_i^*c_i + c_ic_i^* = 1, & 1\leq i\leq N \\
        c_i^*c_j + c_jc_i^* = 0, & 1\leq i\neq j\leq N.
    \end{array}\right.
\end{equation}
It follows that the family \( (s_1, s_2, \ldots, s_N, s_{-1}, \ldots, s_{-N}) \) forms a spin system in \( B(\Lambda_N) \). For each nonempty subset \( A \) of the set \(\{-N, \ldots, -1,1, \ldots,  N\} \), let $k=|A|$ denote its cardinality, and write \[ A = \{i_1 , i_2 , \ldots , i_k\}, \] with $i_1 < i_2 < \ldots < i_k$. We set
\[
s_A := s_{i_1}s_{i_2}\ldots s_{i_k}.
\] Then if \( A = \emptyset \) , we set $$s_A = 1.$$
The family \( \{s_A \mid A \subset \{-N, \ldots, -1,1, \ldots,  N\}\} \) forms a basis for the \( C^* \)-algebra \[C^*\langle s_1, \ldots, s_N, s_{-1}, \ldots, s_{-N} \rangle,\] generated by the spin system. Since both sides have the same dimension, we have
\[
B(\Lambda_N) = C^*\langle s_1, \ldots, s_N, s_{-1}, \ldots, s_{-N} \rangle.
\]
We define the space \(\mathcal{F}_N\) as
\begin{equation}\label{defFn}
    \mathcal{F}_N := \text{span}\left\{1, s_1, \ldots, s_N, s_{-1}, \ldots, s_{-N}, s_{-N}s_N \ldots s_{-1}s_1\right\} \subset B(\Lambda_N),
\end{equation}
and we define \(\mathcal{E}_{2N}\) as \begin{equation}\label{defE2N}
    \mathcal{E}_{2N} := \text{span}\{1,s_1,\ldots,s_N,s_{-1},\ldots,s_{-N}\}\subset B(\Lambda_N).
\end{equation}
Next, we introduce the linear map \(\sigma : \F_N\to\F_N\) defined by 
\begin{equation}\label{defsigma}
   \sigma(1)=1,\quad \sigma(s_j) = s_j,\quad \sigma(s_{-N}s_{N}\cdots s_{-1}s_1) = - s_{-N}s_{N}\cdots s_{-1}s_1
\end{equation}
 for $j\in\{-N,\ldots -1,1,\ldots, N\}.$

Observe that this map is an involution. Moreover, this is an isometry from $\F_N^p$ to itself, where $\F_N^p$ is the subspace $\F_N$ endowed with the $S^p-$norm, for $p\geq 1$. The authors of \cite{MRR} pointed out to me that the explanation given in \cite[Remark 5.5, p.20]{MRR} is not correct for even $N$. However, we may consider the linear map $\tau' : \F_N\to\F_N$ defined by \(\tau'(1)=1,~\tau'(s_j) = -s_j\), for any \(j\in\{-N,\ldots,-1,1,\ldots, N\}\), and \(\tau'( s_{-N}s_{N}\cdots s_{-1}s_1) = - s_{-N}s_{N}\cdots s_{-1}s_1.\) This is an isometry from $\F_N^p$ to itself. Moreover, for the $*$-isomorphism $\pi : B(\La_N)\to B(\La_N)$  taking $s_j$ to $-s_j$ for any $j\in\{-N,\ldots,-1,1,\ldots, N\},$ we have \(\pi\tau' = \tau.\) (see ArXiv version of \cite{MRR}).

\medskip

\begin{definition}\label{type56} A space of even dimension $2N\geq 4$ is called spinorial space if it is equivalent to \[\left\{(x\otimes a_1, \sigma(x)\otimes a_2),~~x\in\F_N^p\right\},\] for some operators $a_i\in S^p(H_i)$, $i=1,2$.

A space of odd dimension $2N+1\geq 5$ is called a spinorial space if it is equivalent to \[\left\{ x\otimes a,~~x\in\E_{2N}^p\right\},\] for an operator $a\in S^p(H_1)$.

\end{definition}

In their memoirs \cite{AF2} and \cite{AF3}, Arazy and Friedman provide a description of indecomposable 1-complemented subspaces for $1\leq p\neq 2<\infty$. Here, we reproduce the statement as it appears in \cite[Theorem 2.8, p.9]{MRR}.

\begin{theorem}[Arazy-Friedman]\label{thmAF}
Let $1\leq p \neq 2 <\infty$, let $H$, $K$ be Hilbert spaces, and let $X$ be an indecomposable subspace of the space $S^p(H,K)$. 
Then $X$ is 1-complemented in $S^p(H,K)$ if, and only if, $X$ is of one of the following types: 
\begin{enumerate}
    \item A space of symmetric matrices, i.e. $X$ is equivalent to the space $\mathcal{S}^p_I\otimes a_1$,
    \item A space of anti-symmetric matrices, i.e. $X$ is equivalent to the space $\mathcal{A}^p_I\otimes a_1$,
    \item A space of rectangular matrices, i.e. $X$ is equivalent to the space $\{(w\otimes a_1, w^\top\otimes a_2),~~ w\in S^p_{I,J}\}$,
    \item An N-dimensional AF-Hilbertian space, i.e. $X$ is equivalent to the space \(\{(\varphi_1(t)\otimes a_1,\ldots, \varphi_N(t)\otimes a_N),~~ t\in\ell^2_N\}\),
    \item A spinorial space of even dimension, i.e. $X$ is equivalent to the space $\{(x\otimes a_1, \sigma(x)\otimes a_2),\quad x\in\F_N^p\}$,
    \item A spinorial space of odd dimension, i.e. $X$ is equivalent to the space $\E_{2N}^p\otimes a_1$,
    
    with $a_i\in S^p(H_i)$ some operators on Hilbert spaces $H_i$, $I$ and $J$ some countable sets, and $N\geq 2$.
\end{enumerate}
\end{theorem}
From now on, type 1 spaces will refer to spaces of symmetric matrices, corresponding to item 1 of the above theorem; the same convention applies to the other numbered items.
\begin{remark}
 In the case of spinorial spaces, in \cite{AF3}, Arazy and Friedman introduced other spaces named \( AH(N) \) and \( DAH(N) \). However, the description given above is equivalent to that of Arazy and Friedman. For further details, see \cite{MRR}, and the Appendix.



\end{remark}

In the sequel, we will discuss some properties of the support of elements or sets. They play an important role in the subsequent work and facilitate the understanding of the key concepts. We will show in Remark \ref{rel_equiv} that the relation introduced in (\ref{defequiv}) is effectively an equivalence relation.
For an operator \( x \in B(H, K) \) on Hilbert spaces, let \( \text{s}_r(x) \) and \( \text{s}_\ell(x) \) denote the right and left supports of \( x \), respectively. By definition, \( \text{s}_r(x) \) (and similarly \( \text{s}_\ell(x) \)) is the smallest orthogonal projection \( e \) in \( B(H) \) (or \( B(K) \)) such that \( x e = x \) (or \( e x = x \)). If \( x \in B(H) \) is self-adjoint, then \( \text{s}_r(x) = \text{s}_\ell(x) \), and this projection is referred to as the support of \( x \), denoted by \( \operatorname{s}(x) \). When \( x \) has the polar decomposition \( x = u |x| \) in \( B(H,K) \), the following identities hold :
\[
\text{s}_r(x) = u^* u \quad \text{and} \quad \text{s}_\ell(x) = u u^*.
\]
More generally, for any subset \( X \) of \( B(H,K) \), we define the left and right supports of \( X \) as :
\[
\text{s}_r(X) := \sup \{ \text{s}_r(x) : x \in X \} \quad \text{and} \quad \text{s}_\ell(X) := \sup \{ \text{s}_\ell(x) : x \in X \}.
\]
If both of these projections are the identity, we say that \( X \) is non-degenerate. These definitions appear in the introduction of \cite{AF2}.

\begin{remark}
 If the subset $X$ is invariant under $*$, that is, $x^*\in X$ for all $x\in X$, then \(\text{s}_\ell(X) = \text{s}_r(X)\). In this case, we denote this projection by \(\operatorname{s}(X)\).
\end{remark}

\medskip
In the following, we will focus on the case where \( X = \text{Ran}(P) \), with \( P : \text{S}^p(H) \to \text{S}^p(H) \) a contractive projection for \( 1 \leq p < \infty \). We will denote the right and left supports of \( P \) by \( \text{s}_r(P) \) and \( \text{s}_\ell(P) \), respectively, which correspond to \( \text{s}_r(\text{Ran}(P)) \) and \( \text{s}_\ell(\text{Ran}(P)) \). It is clear that if the projection is an adjoint preserving map, that is, $P(x^*)=P(x)^*$ for all $x\in S^p(H)$, then Ran$(P)$ is invariant under $*$. Hence, we have \begin{equation}\label{defsP}
    \text{s}_\ell(P) = \text{s}_r(P) := \operatorname{s}(P).
\end{equation} Note that if $P$ is positive, then it is adjoint preserving. 

The definition of equivalent spaces (\ref{defequiv}) can be reformulated by imposing additional conditions on the partial isometries, with the help of the following lemma.

\begin{lemma}\label{equiv_unit}
Let $X\subset S^p(H,K)$, $Y\subset S^p(H',K')$ be closed subspaces. Then $X$ and $Y$ are equivalent if and only if there exist partial isometries $U : H'\to H$ and $V : K'\to K$ such that \begin{equation}\label{equiv_supp}\begin{array}{ccc}
    UU^* = \operatorname{s}_r(X), &\quad& VV^* = \operatorname{s}_\ell(X),  \\
    U^*U = \operatorname{s}_r(Y), &\quad& V^*V = \operatorname{s}_\ell(Y).
\end{array}
\end{equation}
and \[X= VYU^*,\quad Y= V^*XU.\]
\end{lemma}
\begin{proof}
We only need to prove that if $X\sim Y$, we may obtain conditions (\ref{equiv_supp}).
Let \(U : H'\to H\) and \(V : K'\to K\) be two partial isometries such that \[X = VYU^*,\quad Y=V^*XU.\] 
Denote \(r:= \text{s}_r(Y)\) and \(\ell := \text{s}_\ell(Y).\) Note that for each $y\in Y$, $yU^*U = y$, so $r\leq U^*U$. It implies that $\mathcal{U} :=Ur$ is a partial isometry, with $\mathcal{U}^*\mathcal{U} = rU^*Ur = r^2 = r$. Similarly, $\mathcal{V}:=V\ell$ is a partial isometry with $\mathcal{V}^*\mathcal{V} =\ell$, and we have \[X = \mathcal{V}Y\mathcal{U}^*,\quad Y = \mathcal{V}^*X\mathcal{U}.\]
Now, we show that \[\mathcal{V}\mathcal{V}^* = \text{s}_\ell(X),\quad \mathcal{U}\mathcal{U}^* = \text{s}_r(X).\]
Since these maps are orthogonal projections, to prove an equality, it suffices to show that they have the same range. The reasoning is detailed below for $\mathcal{U}$ only.
For every $x\in X$, $x\mathcal{U}\mathcal{U^*} = x$, so \(\text{s}_r(X) \leq \mathcal{U}\mathcal{U^*},\) and \(\text{s}_r(X)\mathcal{U}\) is a partial isometry. Then \(Y = \mathcal{V^*}X\text{s}_r(X)\mathcal{U}\), so \[r\leq (\text{s}_r(X)\mathcal{U})^*(\text{s}_r(X)\mathcal{U}) = \mathcal{U}^*\text{s}_r(X)\mathcal{U} \leq \mathcal{U}^*\mathcal{U}=r.\] Hence \[\mathcal{U}^*\text{s}_r(X)\mathcal{U}=\mathcal{U}^*\mathcal{U}.\] Since \(\Ran(\text{s}_r(X))\subset \Ran(\mathcal{U})\), then \[\mathcal{U}\mathcal{U}^*\text{s}_r(X)\mathcal{U} =\text{s}_r(X)\mathcal{U},\] and we deduce \[ \mathcal{U}=\text{s}_r(X)\mathcal{U}.\] Hence \(\Ran(\mathcal{U}) \subset \Ran(\text{s}_r(X)).\)

\end{proof}
As a consequence of Lemma \ref{equiv_unit}, we see that in the definition of equivalent spaces (\ref{defequiv}), we may assume that the partial isometries are unitaries if we have $X$ and $Y$ non-degenerate.

\begin{remark}\label{rel_equiv}
\begin{enumerate}
\item We now can prove that the relation $\sim$ defined in (\ref{defequiv}) is an equivalence relation. Let $X\subset S^p(H,K)$, $Y\subset S^p(H',K')$ and $Z\subset S^p(H'',K'')$ be closed subspaces such that \[X\sim Y\quad \text{ and }\quad Y\sim Z.\]
Then by Lemma \ref{equiv_unit}, we have partial isometries \[\begin{array}{cc}
    U_1 : H'\to H & U_2 : H''\to H' \\
    V_1 : K'\to K & V_2 : K'' \to K'
\end{array}\] such that \[\begin{array}{cc}
    X = V_1YU_1^*, & Y = V_1^*XU_1 \\
    Y = V_2ZU_2^* & Z = V_2^*YU_2.
\end{array}\] and \[\text{s}_r(X) = U_1U_1^*,\quad \text{s}_\ell(X) = V_1V_1^*,\quad \text{s}_r(Y) = U_1^*U_1 = U_2U_2^*,\]\[ \text{s}_\ell(Y) = V_1^*V_1 = V_2V_2^*,\quad \text{s}_r(Z) = U_2^*U_2,\quad \text{s}_\ell(Z) = V_2^*V_2.\]
Then \[X = V_1V_2Z(U_1U_2)^*,\quad Z = (V_1V_2)^*X(U_1U_2)\]
hence it suffices to check that $U_1U_2$ and $V_1V_2$ are partial isometries. We can compute \[(U_1U_2)(U_1U_2)^*(U_1U_2) = U_1U_2U_2^*U_1^*U_1U_2 = U_1U_1^*U_1U_1^*U_1U_2 = U_1U_2.\]
The computation is the same for \(V_1V_2\). This shows that the relation $\sim$ is transitive, the rest follows easily.

\item If \(X\) is a subspace of \(S^p(H,K)\), we define the subspaces $\H:= \Ran(\text{s}_r(X)) \subset H$, and $\K:= \Ran(\text{s}_\ell(X))\subset K$. It is clear that $X\subset S^p(\H,\K)$ is non-degenerate and equivalent to $X\subset S^p(H,K)$. Lemma \ref{equiv_unit} also allows us to establish that if we have two equivalent spaces $X$ and $Y$, $X$ is indecomposable if and only if $Y$ is.

\end{enumerate}
\end{remark}

The following lemma allows us to restrict, in Definition \ref{type123}, Definition \ref{type4} and Definition \ref{type56}, to the case where the operators $a$, $a_i$ are positive and injective, which slightly simplifies the description of Theorem \ref{thmAF}. For example, let $I$ be a countable index set, and let $a\in S^p(H_1)$. If a subspace $X$ of the space $S^p(H)$ is equivalent to \(\S_I^p\otimes a\), then there exists a positive and injective operator $\alpha \in S^p(K_1)$, such that $X$ is equivalent to \(\S_I^p\otimes \alpha\).

\begin{lemma}
Let $Z$ be a Banach space, $n\in\Ndb$ an integer, and for $i=1,\ldots ,n$, let $H_i$, $K_i$, $H_i'$ be Hilbert spaces. Assume that $\psi_i : Z\to S^p(H_i,K_i)$ is an isometry, and $a_i\in S^p(H_i')\setminus\{0\}$. Set \(H_i''= \operatorname{Ker}(a_i)^\perp\), and denote by $\alpha_i: H_i''\to H_i''$ the restriction of the operator $|a_i|$. Then the two subspaces \[X_1 :=\{(\psi_1(z)\otimes a_1,\psi_2(z)\otimes a_2,\ldots ,\psi_n(z)\otimes a_n),\quad z\in Z\}\subset \bigoplus^p_{i=1,\ldots ,n} S^p(H_i\overset{2}{\otimes} H_i', K_i\overset{2}{\otimes} H_i')\] and \[ X_2 :=\{(\psi_1(z)\otimes \alpha_1,\psi_2(z)\otimes \alpha_2,\ldots ,\psi_n(z)\otimes \alpha_n),~ z\in Z\}\subset \bigoplus^p_{i=1,\ldots ,n} S^p(H_i\overset{2}{\otimes} H_i'', K_i\overset{2}{\otimes} H_i'')\] are equivalent.
\end{lemma}
\begin{proof}
We first prove the following equivalence relation, \[X_1\sim  Y:=\{(\psi_1(z)\otimes |a_1|,\psi_2(z)\otimes |a_2|,\ldots ,\psi_n(z)\otimes |a_n|),\quad z\in Z\}.\]
For $1\leq i\leq n$, denote by \(a_i = u_i|a_i|\) the polar decomposition of $a_i$. Then, we define a partial isometry as follows \[U:= (\operatorname{Id}_{K_1}\otimes u_1,\operatorname{Id}_{K_2}\otimes u_2,\ldots ,\operatorname{Id}_{K_n}\otimes u_n) \in B(\bigoplus^2_{1\leq i\leq n} K_i\overset{2}{\otimes} H_i').\]  To prove the equivalence relations, observe that \[UY = X_1,\quad  \text{ and }\quad U^*X_1 = Y.\]

Now, denote by $s_i : H_i'\to H_i''$ the orthogonal projection onto $H_i''$. Then $s_is_i^* = \operatorname{Id}_{H_i''}$, and $s_i^*s_i = \operatorname{s}(|a_i|)$. Using the identities \[\alpha_i = s_i|a_i|s_i^*,\quad |a_i| = s_i^*\alpha_is_i,\] and the partial isometries  \(V:= (\operatorname{Id}_{K_1}\otimes s_1,\ldots ,\operatorname{Id}_{K_n}\otimes s_n)\) and \( W:=(\operatorname{Id}_{H_1}\otimes s_1, \ldots ,\operatorname{Id}_{H_n}\otimes s_n),\) we obtain

\[ VYW^* = X_2,\quad \text{ and }\quad V^*X_2W = Y.\]
Hence \(Y\sim X_2\), and therefore \(X_1\sim X_2\).
\end{proof}


\section{Results on $JC^*$-triples}\label{sectionJC*triple}

The definition and basic properties of $JC^*$-triples are given by Harris in \cite{H2} and \cite{H}, and reviewed in this section.
A $JC^*$-triple is a closed subspace $J$ of $B(H)$, where $H$ is a Hilbert space, such that \(xx^*x\in J\) for each $x\in J$. Using polarisation identities, we obtain that a $JC^*$-triple $J\subset B(H)$ is closed under the triple product \begin{equation}\label{tripleproduct}
     \{x,y,z\} = \dfrac{xy^*z + zy^*x}{2}.
\end{equation}
For example, the space of symmetric matrices $\S_I$ in $B(\ell^2_I)$, or the space of anti-symmetric matrices $\A_I $ in $ B(\ell^2_I)$ are $JC^*$-triples. It is also important to see that the spaces $\F_N$ and $\E_{2N}$ introduced in (\ref{defFn}) are $JC^*$-triples. This result is a consequence of the following proposition.

\begin{proposition}
Let $(w_1,\ldots w_m)$ be a spin system in $B(H)$. Then the two subspaces \[\operatorname{span}\{w_1,\ldots ,w_m\},\quad \text{ and }\quad \operatorname{span}\{1,w_1,\ldots ,w_m\}\] are $JC^*$-triples. 
\end{proposition}

\begin{proof}
First of all, these two subspaces are of finite dimension, so they are closed. Let \(X = \text{span}\{w_1,\ldots ,w_m\},\) we prove that $X$ is closed under the triple product defined above. By linearity, it suffices to prove that for every \(i,j,k\in\{1,\ldots,m\},\) the element \(2\left\{w_i,w_j,w_k\right\} = w_iw_jw_k + w_kw_jw_i\in X.\) We use the relations \(w_iw_j = -w_jw_i\) for $1\leq i\neq j \leq m$, and $w_j^* = w_j = w_j^{-1}$, for each $1\leq j \leq m$.

Let \(i,j,k\in\{1,\ldots,m\},\) and let $x = w_iw_jw_k + w_kw_jw_i$. Then \begin{enumerate}
    \item if $i,j,k$ are all distinct, \(x = w_k(w_iw_j+w_jw_i) = 0\in X.\)
    \item If $i=j$, \(x = 2w_k \in X.\)
    \item If $j=k$, \(x = 2w_i\in X.\)
    \item If $i=k$, \(x = \left|\begin{array}{cc}
        -2w_j \in X & \text{ if }j\neq i, \\
        2w_j \in X & \text{otherwise.}
    \end{array}\right.\)
\end{enumerate}
This shows that $X$ is a $JC^*$-triple. 

For the second one, we prove that it is closed under the mapping $x\mapsto xx^*x$. It is sufficient to show that for all $y\in \text{span}\{w_1,\ldots ,w_m\},$ \[(1+y)(1+y)^*(1+y) \in \text{span}\{1,w_1,\ldots ,w_m\}.\]
Let $y\in \text{span}\{w_1,\ldots ,w_m\}$, we have \[(1+y)(1+y)^*(1+y) = \underbrace{1+2y+y^*+yy^*y}_{\in  \text{span}\{1,w_1,\ldots ,w_m\}} + y^*y+yy^*+y^2.\]
We may write $ y = \sum\limits_{k=1}^m \alpha_k w_k \in Y$, then \[y^*y+yy^*+y^2 = \sum_{k,l=1}^m \underbrace{(\overline{\alpha}_k\alpha_l + \alpha_k\overline{\alpha}_l + \alpha_k\alpha_l)}_{\text{symmetric in $k$ and $l$}}w_kw_l,\]
hence \[2(y^*y+yy^*+y^2) = \sum_{k,l=1}^m (\overline{\alpha}_k\alpha_l + \alpha_k\overline{\alpha}_l + \alpha_k\alpha_l)\underbrace{( w_kw_l + w_lw_k)}_{=2\delta_{kl}.1} = 2\sum_{k=1}^m (2|\alpha_k|^2 + \alpha_k^2).1, \] which belongs to \(\text{span}\{1,w_1,\ldots ,w_m\}.\)
\end{proof}

Note that for the spin system $(s_1,\ldots s_N,s_{-1},\ldots s_{-N})$ defined in (\ref{defs_j}), the family \[(s_1,\ldots s_N,s_{-1},\ldots s_{-N}, i^ns_{-N}s_N\cdots s_{-1}s_1)\] is also a spin system. It implies that the spaces $\F_N$ and $\E_{2N}$ are $JC^*$-triples.

\begin{definition}
Let $J\subset B(H_1)$, $L\subset B(H_2)$ be $JC^*$-triples. A bounded linear  map $\varphi : J\to L$ that verifies \[\varphi(xx^*x) = \varphi(x)\varphi(x)^*\varphi(x),\quad x\in J,\] is called a triple-homomorphism. A bijective triple-homomorphism is called a triple isomorphism.
\end{definition}
Note that such a map preserves the triple product defined in (\ref{tripleproduct}). Moreover, a triple-homomorphism is contractive, and a triple-isomorphism is isometric. 
A subspace of $B(H)$ is called unital if it contains $1=\operatorname{Id}_H.$

\begin{lemma}\label{JCautoadj}
Let $J\subset B(H)$ be a unital $JC^*$-triple, then \begin{enumerate}
    \item $J$ is invariant under $*$, i.e. for every $x\in J$, $x^*\in J$.
    \item Let $J'$ be another unital $JC^*$-triple, and let $\varphi : J\to J'$ be a triple-homomorphism. If $\varphi$ is unital, that is $\varphi(1)=1$, then $\varphi$ is adjoint preserving: for every $x\in J,$ $\varphi(x^*) = \varphi(x)^*.$
\end{enumerate}
\end{lemma}
\begin{proof}
The two results immediately follow from the identity \[x^* = \{1,x,1\}.\]
\end{proof}

\begin{proposition} Let $J\subset B(H)$ be a unital $JC^*$-triple, and let $\varphi : J\to J$ be a unital triple-homomorphism.
Then, for any unitary element $u\in J$, $\varphi(u) \in J$ is also unitary.
\end{proposition}
\begin{proof}This results follows \cite[(a) of Proposition 2.1, p.335]{H2}, applied  with $A=1$. \end{proof} 

\begin{proposition}\label{propfunctiononJC}
Let $J\subset B(H)$, $J'\subset B(K)$ be $JC^*$-triples, and let $\varphi : J\to J'$ be a triple-homomorphism. For $x\in J$, let $x = u|x|$ and $\varphi(x) = w|\varphi(x)|$ denote the polar decompositions of $x$ and $\varphi(x)$, respectively. Let $f : [0,\|x\|]\to \Cdb$ be a continuous function with $f(0)=0$. Then:
\begin{enumerate}
    \item  \(uf(|x|)\in J,\) where $f(|x|)$ is defined by the continuous functional calculus,
    \item  \(\varphi\left(uf(|x|)\right) = wf(|\varphi(x)|).\)
\end{enumerate}
\end{proposition}
\begin{proof}
Note that \(xx^*x = u|x|^3 \in J \text{ and } x(xx^*x)^*x = u|x|^5 \in J.\)
By induction we obtain \[ u|x|^{2n+1} \in J,\quad  n\geq 0.\]
Thus, \[uP(|x|)\in J,\] for every odd polynomial $P\in\Cdb[X]$.
Since any continuous function $f$ on the spectrum $[0,\|x\|]$, with $f(0)=0$ can be uniformly approximated by odd polynomials, we obtain the first stated result.

We also have \[\varphi(u|x|^3) = \varphi(x)\varphi(x)^*\varphi(x) = w|\varphi(x)|^3,\] and by indcution, \[\varphi(u|x|^{2n+1}) = w|\varphi(x)|^{2n+1},\quad n\geq 0.\] We obtain \[\varphi(uP(|x|)) = wP(|\varphi(x)|),\] for $P\in \Cdb[X]$ an odd-polynomial, and we can deduce the desired result.
\end{proof}
 
\section{Description of positively 1-complemented subspaces of $S^p$}\label{sectionDescriposit}

We begin this section with a discussion on the support of a positive projection.

\begin{remark}\label{support_positive_proj} Let $1\leq p<\infty$, let $H$ be a Hilbert space, and let $P : S^p(H)\to S^p(H)$ be a positive contractive projection. Recall that $\text{s}(P)$ is defined as follows: \[\text{s}(P) = \sup\{\operatorname{s}_r(x),\quad x\in \Ran(P)\} = \sup\{\operatorname{s}_\ell(x),\quad x\in \Ran(P)\}.\]
Noting that $\Ran(P)$ is generated by its positive elements, it is easy to show that \begin{equation*}
    \operatorname{s}(P) = \sup\{\operatorname{s}(x),\quad x\geq 0,~x\in \Ran(P)\}.\end{equation*}
One should also observe that although a contractive projection $Q$ may be adjoint preserving, with $\text{s}_r(Q)  = \text{s}_\ell(Q)$, this support projection is not necessarily equal to $\sup\{\operatorname{s}(x),~x\geq 0,~x\in \Ran(Q)\}$. For example, the set of all anti-symmetric complex matrices of size $n\in\Ndb$ (or more generally $\A_I$) is invariant under $*$ and non-degenerate, yet contains no non-zero positive elements.
\end{remark}

In fact, Arhancet and Raynaud in \cite[Section 6]{AR} proved a more precise result, which captures a fundamental structural property of positive contractive projections.

\begin{proposition}\label{propAR}
Let $1\leq p <\infty$, let $H$ be Hilbert space, and let $P : S^p(H)\to S^p(H)$ be a positive contractive projection. Then there exists a positive element $x$ of $\operatorname{Ran}(P)$ such that \[\operatorname{s}(x) = \operatorname{s}(P).\]
\end{proposition}

The next lemma is well-known (and is implicitly used in the proof of \cite[Proposition 6.2, p.873]{MRR}).
\begin{lemma}\label{proj_sur_blocs_diago}
Let $N\geq 2$ let $H_1,\ldots, H_N$ be Hilbert spaces. We denote $H:= H_1\overset{2}{\oplus}H_2\overset{2}{\oplus}\cdots\overset{2}{\oplus} H_N$. The subspace \(\underset{i=1,\ldots, N}{\overset{p}{\bigoplus}} S^p(H_i)\) of the space $S^p(H)$ is 1-complemented. 
\end{lemma}
\begin{proof} 
Any $x\in S^p(H)$ may be written as $x = \left[ x_{kl}\right]_{1\leq k,l\leq N}$, with $x_{kl} \in S^p(H_l,H_k)$. Note that if $s_k:H\to H_k$ denote the orthogonal projection onto $H_k$ for any $1\leq k\leq N$, then $x_{kl} = s_kxs_l^*.$ For any $\theta \in \Rdb$, we consider the diagonal matrix $$D_\theta := \begin{bmatrix}
e^{i\theta} id_{H_1} & & & (0) \\
& e^{2i\theta} id_{H_2} & \\
& & \ddots & \\
(0)& & & e^{iN\theta} id_{H_N}
\end{bmatrix} \in \underset{i=1,\ldots, N}{\overset{\infty}{\bigoplus}} B(H_i).$$ We have \[D_\theta x D_\theta^* = \left[e^{i(k-l)\theta}x_{kl}\right],\] then define \[P(x) := \begin{bmatrix}
x_{11} & & & (0) \\
& x_{22} & \\
& & \ddots & \\
(0)& & & x_{NN}
\end{bmatrix}.\]The map $P:S^p(H)\to S^p(H)$ is a projection onto $\underset{i=1,\ldots, N}{\overset{p}{\bigoplus}} S^p(H_i)$. Since we can write \[P(x) = \frac{1}{2\pi}\int_{-\pi}^\pi D_\theta x D_\theta^*~ d\theta,\] we deduce \[\|P(x)\|_p = \left\|\frac{1}{2\pi}\int_{-\pi}^\pi D_\theta x D_\theta^* ~d\theta\right\|_p \leq \frac{1}{2\pi}\int_{-\pi}^\pi \|D_\theta xD_\theta^*\|_p ~d\theta = \frac{1}{2\pi}\int_{-\pi}^\pi \|x\|_p ~d\theta = \|x\|_p.\] We deduce that $P$ is contractive.
\end{proof}

In the study of positively 1-complemented subspaces, it is natural to focus on indecomposable components. The following lemma formalizes this idea.

\begin{lemma}\label{decpos1compl}
Let $1< p<\infty$, $H$ a Hilbert space, and let $X$ be a subspace of $S^p(H)$. Consider any decomposition \[X = \bigoplus^p_\alpha X_\alpha,\] with a family of indecomposable pairwise operator-disjoint subspaces $X_\alpha \subset S^p(H)$.
Then $X$ is positively 1-complemented if, and only if, every $X_\alpha$ is positively 1-complemented.
\end{lemma}
\begin{proof}
Assume that the subspace $X$ of the space $S^p(H)$ is positively 1-complemented, then every $X_\alpha$ is 1-complemented. We need to show that it is positively 1-complemented. Set \[H':= \Ran(\operatorname{s}(X)) \quad r_\alpha := \text{s}_r(X_\alpha), \quad l_\alpha := \text{s}_\ell(X_\alpha), \quad H_\alpha := \Ran(r_\alpha),\quad K_\alpha :=\Ran(l_\alpha).\] 
By Proposition \ref{propAR}, there exists a positive element $x\in X$, such that \(\operatorname{s}(x) = \operatorname{s}(X).\)
Then for each $\alpha$, there exists $x_\alpha \in X_\alpha$ such that \[x = \sum_\alpha x_\alpha.\]
Then, we have \(x_\alpha = l_\alpha x.\) Since $K_\alpha \subset H'$ and $l_\alpha$ is the orthogonal projection onto $K_\alpha$, it follows that 
\[\overline{\Ran(x_\alpha)} = \overline{l_\alpha(H')} = K_\alpha.\]
Thus, we obtain that \[\text{s}_\ell(x_\alpha) = l_\alpha.\]
Similarly, \[\text{s}_r(x_\alpha) = r_\alpha.\]
Furthermore, since the elements $x_\alpha$ and $x_\beta$ are disjoint for $\alpha \neq \beta$, we have \[\sum_\alpha x_\alpha = x = |x| = \sum_\alpha |x_\alpha|,\]
which implies $x_\alpha = |x_\alpha|$. Hence each \(x_\alpha\) is positive, and \[r_\alpha = l_\alpha := s_\alpha.\]
It follows that \(X_\alpha = s_\alpha X s_\alpha\), so this subspace is positively 1-complemented.

For the converse, suppose that
\[
X = \bigoplus^p_\alpha X_\alpha,
\]
where each \(X_\alpha \subset S^p(H_\alpha)\) is a positively 1-complemented subspace, with \(H_\alpha = \operatorname{Ran}(\operatorname{s}(X_\alpha))\), and let
\[
P_\alpha : S^p(H_\alpha) \to S^p(H_\alpha)
\]
be the corresponding contractive positive projection onto \(X_\alpha\). Let \(s_\alpha : H \to H_\alpha\) denote the orthogonal projection onto $H_\alpha$. Then \[P : x\in S^p(H) \mapsto \sum_\alpha P_\alpha(s_\alpha x s_\alpha^*)\] is the contractive projection onto $X$, and the latter is positive. Indeed, for any $x\in S^p(H)$, and any finite set of indices $A$, denote $S_A := \sum\limits_{\alpha\in A} s_\alpha$, the orthogonal projection onto the space $$\bigoplus^2_{\alpha\in A} H_\alpha \subset H.$$ Since the elements $P_\alpha(s_\alpha xs_\alpha^*)$ and $P_\beta(s_\beta xs_\beta ^*)$ (resp. $s_\alpha xs_\alpha^*$ and $s_\beta xs_\beta ^*$) are disjoint for $\alpha\neq \beta$, and since each $P_\alpha$ is contractive, we have \[\left\|\sum_{\alpha\in A}P_\alpha(s_\alpha x s_\alpha^*)\right\|_p^p = \sum_{\alpha \in A}\|P_\alpha(s_\alpha x s_\alpha^*)\|_p^p \leq \sum_{\alpha \in A}\|s_\alpha x s_\alpha^*\|_p^p = \left\|\sum_{\alpha \in A}s_\alpha x s_\alpha^*\right\|_p^p.\] Using Lemma \ref{proj_sur_blocs_diago} with the embedding $\underset{\alpha\in A}{\overset{p}{\bigoplus}} S^p(H_\alpha) \subset S^p\left(\underset{\alpha\in A}{\overset{2}{\bigoplus}} H_\alpha\right)$, we obtain \[\left\|\sum_{\alpha \in A}s_\alpha x s_\alpha^*\right\|_p^p \leq \left\|S_A xS_A^*\right\|_p^p \leq \|x\|_p^p.\] We deduce, using Cauchy's criteria, that for any $x\in S^p(H)$, the sum $\sum\limits_\alpha P_\alpha(s_\alpha x s_\alpha^*)$ converges in $S^p(H)$ and that the resulting map $P : S^p(H)\to S^p(H)$ is contractive. The positivity of $P$ follows that fact that each $P_\alpha$ and each $x\mapsto s_\alpha x s_\alpha ^*$ is positive.

\end{proof}


\medskip

Let $H$ and $H'$ be Hilbert spaces, and let $1\leq p<\infty$. We say that two closed subspaces $X\subset S^p(H)$ and $Y\subset S^p(H')$, which are invariant under $*$, are \textit{positively} equivalent, denoted by \begin{equation*}
    X\simp Y,
\end{equation*} if there exists a partial isometry $U : H'\to H$ such that \begin{equation}
    X = UYU^* \quad \mbox{ and }\quad Y = U^*XU.
\end{equation}

In this case, $X$ is positively 1-complemented if, and only if $Y$ is. To see this, just apply (\ref{proj_equiv_spaces}). As for the relation $\sim$ (see Lemma \ref{equiv_unit}), we can choose the map $U$ such that \[UU^* = \text{s}_r(X) = \text{s}_\ell(X),\quad U^*U = \text{s}_r(Y) = \text{s}_\ell(Y).\] As in Remark \ref{rel_equiv}, we can show that \(\simp\) is an equivalence relation.

\medskip

We now turn our attention to the classification of positively 1-complemented, indecomposable subspaces. Theorem \ref{THM1} below will be proved in the next two sections. This will be achieved through a blend of analytical and synthetic methods. For each of the six types of indecomposable 1-complemented subspaces, we examine which ones are also positively 1-complemented. An important part of Theorem \ref{THM1} is that there is no AF-Hilbertian positively 1-complemented subspace (see type 4 in Theorem \ref{thmAF}).

\begin{theorem}\label{THM1}
Let $1\leq p <\infty$, let $H$ be a Hilbert space, and let $X$ be an indecomposable subspace of the space $S^p(H)$ . Then $X$ is positively 1-complemented in $S^p(H)$ if, and only if, $X$ is positively equivalent to one of the following spaces : 

- The space $O\S^p_I \otimes a$, where $O\in \S_I$ is a symmetric unitary operator,

- The space $O\A^p_I \otimes a$, where $O\in \A_I$ is an anti-symmetric unitary operator,

- The space $\{(w\otimes a_1,w^\top\otimes a_2),\quad w\in S^p_I\}$,

- The space $\{(vx\otimes a_1, \sigma(v)\sigma(x)\otimes a_2),\quad x\in \F_N^p\}$, where $v\in\F_N$ is a unitary operator,

- The space $v\E_{2N}^p\otimes a$, where $v\in\E_{2N}$ is a unitary operator,

with $a,a_1\in S^p(H_1)$, $a_2\in S^p(H_2)$ some positive injective operators on Hilbert spaces $H_1,~H_2$, $I$ a countable set, and $N\geq 2$.
\end{theorem}

\section{Proof of Theorem \ref{THM1}, for $p\neq 2$.}\label{sectionProofpneq2}

This section is mostly devoted to prove Theorem \ref{THM1} for $1<p\neq 2<\infty$. The case $p=1$ will be treated at the end of this section. We rely on Theorem \ref{thmAF} (see also the two lines following it) and Lemma \ref{linear_form_cp} that we will use repeatedly.

\begin{remark}\label{rem_before_proof}
\begin{enumerate}
    \item Let $1< p\neq 2<\infty$, and $X$ be an indecomposable positively 1-complemented subspace of $S^p(H)$. Let $H'$ be the space $\Ran(\operatorname{s}(X))$, and let $s\in B(H,H')$ be the map induced by $\operatorname{s}(X).$ Then $X$ is positively equivalent to $sXs^* \subset S^p(H')$, and this space is non-degenerate. This shows that it is sufficient to describe the indecomposable \textit{non-degenerate} positively 1-complemented subspace of $S^p(H)$.
    \item It is also important to observe that, according to Lemma \ref{equiv_unit}, if $X\subset S^p(H)$, and $Y\subset S^p(K)$ are non-degenerate and equivalent subspaces, then there exist \textit{unitaries} $U,V : K\to H$ such that \(X = VYU^*.\)
\end{enumerate}
\end{remark}

Here is the strategy of the proof. We consider $X$ an indecomposable non-degenerate and positively 1-complemented subspace of the space $S^p(H)$, for $1< p\neq 2<\infty$. According to Theorem \ref{thmAF}, this subspace $X$ is equivalent to a subspace $Z$, where $Z$ is in one of the 6 distinct forms described in Theorem \ref{thmAF}. For each case, we exploit the positivity of the contractive projection to obtain a more precise description of $X$.  
 
\subsection{Type 1.} 
Here, we consider the case where $X\subset S^p(H)$ is of type 1 (see Definition \ref{type123}), and non-degenerate. Assume that $X$ is positively 1-complemented. Note that if $I$ is a countable set, the space $\S_I^p \subset S^p_I$ is non-degenerate; moreover, if $a\in S^p(H_1)$ is a positive injective operator, then the space $\S_I^p\otimes a\subset S^p(\ell^2_I(H_1))$ is also non-degenerate. Using Remark \ref{rem_before_proof} (2), we can describe $X$ as follows.
There exist $a\in S^p(H_1)$ a positive injective operator on a Hilbert space $H_1$, $I$ a countable index set, and two unitary operators $U$ and $V$ :  $\ell^2_I(H_1)\to H$, such that \begin{equation}
    X = \left\{V(w\otimes a)U^*,\quad w\in \S^p_I\right\}.
\end{equation}
By Proposition \ref{propAR}, there is a positive element $x\in X$, such that $\operatorname{s}(x) = \operatorname{s}(P) = 1$. We write \(x\) as \[ x = V(w_0\otimes a)U^*,\] where $w_0\in \S^p_I$. Then $w_0$ is an injective symmetric operator with a dense range.
Using the identity \(x = |x|,\) and writing \[x^*x = U(w_0^*\otimes a)V^*V(w_0\otimes a)U^* = U(w_0^*w_0\otimes a^2)U^* = (U(|w_0|\otimes a)U^*)^2,\]
we obtain that \[|x| = U(|w_0|\otimes a)U^*,\] and hence that \[V(w_0\otimes a)U^* = U(|w_0|\otimes a)U^*.\]
Let $w_0 = O|w_0|$ be the polar decomposition of $w_0$, where $O\in B(\ell_I^2)$ is a unitary operator. Then \[V(O\otimes 1)(|w_0|\otimes a) = U(|w_0|\otimes a).\] By the uniqueness of polar decomposition of \(U(|w_0|\otimes a)\), we deduce that \[V(O\otimes 1) = U.\]
Furthermore, we can compute that $|w_0^\top| = (O^\top)^*|w_0|^\top O^\top$, since both side are positive and have the same square. Since $w_0\in\S^p_I$,\begin{equation}\label{decpoltranspose}
    w_0 = w_0^\top = |w_0|^\top O^\top= O^\top(O^\top)^*|w_0|^\top O^\top = O^\top|w_0^\top| = O^\top|w_0|^\top,
\end{equation} hence \[w_0 = O^\top|w_0|\] is the polar decomposition of $w_0$. We deduce, by the uniqueness of polar decomposition of $w_0$, that $O$ is a symmetric operator. Moreover, since $\top$ is an anti-$*$-isomorphism, $O^*$ is also symmetric. We may therefore replace $O$ by $O^*$.

To summarize, there is a symmetric and unitary operator $O \in B(\ell^2_I)$, such that \begin{equation}
    X = \left\{U(z\otimes a)U^*,\quad z\in O\S^p_I\right\}
\end{equation}
In other words, $X$ is positively equivalent to the space $O\S^p_I\otimes a$.

Conversely, suppose that $X \subset S^p(H)$ is positively equivalent to the space $O\S^p_I\otimes a$, where $I$ is a countable set, $a\in S^p(H_1)$ is a positive injective operator on a Hilbert space $H_1$, and $O\in \S_I$ is a symmetric unitary operator. We may assume that $\|a\|_p=1$. Since \( X \) is positively 1-complemented if and only if \( O \mathcal{S}_I^p \otimes a \) is, we may focus on determining the contractive projection onto the latter space.
We define $P_{O} : S^p_I\to S^p_I$ as the contractive projection onto the subspace $O\S^p_I$. Since the contractive projection onto $\S^p_I$ is given by $x\mapsto \dfrac{x + x^\top}{2}$, the map $P_{O}$ can be expressed as

\[P_{O} : x\in S^p_I\mapsto \dfrac{x+ Ox^\top O^*}{2},\]
which is a positive and contractive map.
Let $\varphi_a : S^p(H_1)\to \Cdb$ denote the norm one positive linear functional defined by \begin{equation}\label{defvarphia}
    \varphi_a(x) = \Tr(a^{p-1}x).
\end{equation}
According to Lemma \ref{linear_form_cp} and Remark \ref{Totimesphi} (2), the map $P_{O}\otimes\varphi_a : S^p(\ell^2_I(H_1))\to S^p_I$ is positive and contractive. Finally, the contractive projection onto the subspace $O\S^p_I\otimes a$ is given by the map \[P : x\mapsto (P_{O}\otimes\varphi_a)(x)\otimes a,\]
hence is a positive map.

\subsection{Type 2.}
We now consider the case where $X$ is a non-degenerate subspace of type 2 of $S^p(H)$ (see Definition \ref{type123}). Assume that $X$ is positively 1-complemented.  As before, the subspace $\A_I^p$ of $S^p_I$ is non-degenerate, and we have the following description:
There exist $a\in S^p(H_1)$ a positive injective operator on a Hilbert space, and unitary operators $U$ and $V$ :  $\ell^2_I(H_1)\to H$, such that \begin{equation}
    X = \left\{V(w\otimes a)U^*,\quad w\in \A^p_I\right\}.
\end{equation}
As in the previous case, there exists a unitary operator $O\in B(\ell^2_I)$ such that \[X = \left\{U(z\otimes a)U^*,\quad z\in O^*\A^p_I\right\}.\]

We now show that $O\in\A_I.$
Indeed, by construction, $O$ is the unitary operator in the polar decomposition of an element $w_0\in \A^p_I$, that is, $w_0 = O|w_0|$.
By the same calculations as before (see (\ref{decpoltranspose})), the polar decomposition of $w_0^\top$ is \begin{equation}\label{polardecsigma}
    w_0^\top = O^\top|w_0|,
\end{equation}
and since \[w_0^\top = -w_0,\]
it follows that $O$ is anti-symmetric. As in type 1, we may replace $O$ by $O^*$. Therefore, we obtain \[X\simp O\A_I^p\otimes a,\]
with a unitary $O\in \A_I$.

Conversely, consider \(X\simp O\A_I^p\otimes a\), with $I$ a countable set, $O\in \A_I$ unitary, and $a\in S^p(H_1)$ a positive injective operator. We may assume \(\|a\|_p=1\). Using the same positive linear map $\varphi_a$ as defined in (\ref{defvarphia}), and letting $P_{O} : S^p_I\to S^p_I$ denote the contractive projection onto $O\A^p_I$, the contractive projection onto $O\A^p_I\otimes a$ is given by \[P : x\mapsto (P_{O}\otimes\varphi_a)(x)\otimes a.\]
Since the contractive projection onto $\A^p_I$ is given by $x\mapsto \dfrac{x - x^\top}{2}$, and noting that $O^\top=-O$, we have the explicit form
\[P_{O} : x\in S^p_I\mapsto \dfrac{x+ Ox^\top O^*}{2},\]
which is a positive map, and therefore \(P\) is positive as well.

\subsection{Type 3.}
Assume that $X$ is a non-degenerate subspace of type 3 of $S^p(H)$. Then there exist countable index sets $I$ and $J$, Hilbert spaces $H_1$ and $H_2$ and injective and positive operators $a_1\in S^p(H_1)$, $a_2\in S^p(H_2)$ such that \begin{equation}\label{type3sansred}
    X\sim \underbrace{\left\{(w\otimes a_1, w^\top\otimes a_2),\quad w\in S^p_{I,J}\right\}}_{:=Y}.
\end{equation}
Assume, in addition, that $X$ is positively 1-complemented. We suppose that $a_1\neq 0$ and $a_2\neq 0$. The case when $a_1=0$ or $a_2=0$ is similar and will be omitted.
Since $S^p_{I,J}$ and $S^p_{J,I}$ are non-degenerate, the subspace $Y$ of the space $S^p\left(\ell^2_J(H_1)\overset{2}{\oplus}\ell^2_I(H_2),\ell^2_I(H_1)\overset{2}{\oplus}\ell^2_J(H_2)\right)$ is non-degenerate. So we have two unitary operators \[U : \ell_J^2(H_1)\overset{2}{\oplus}\ell^2_I(H_2)\to H,\quad V : \ell_I^2(H_1)\overset{2}{\oplus}\ell^2_J(H_2)\to H,\]
such that \[X = \{V(w\otimes a_1, w^\top\otimes a_2)U^*,\quad w\in S^p_{I,J}\}.\]
For $i=1,2$, we set \[\begin{array}{cc}
    U_1 = U(\cdot,0) : \ell^2_J(H_1)\to H,\quad & V_1 = V(\cdot,0) : \ell^2_I(H_1)\to H  \\
    U_2 = U(0,\cdot) : \ell^2_I(H_2)\to H,\quad & V_2 = V(0,\cdot) : \ell^2_J(H_2)\to H
\end{array}\]
Then, if we let \[\Ran(U_1) = H',\quad \Ran(U_2) = H'',\quad \Ran(V_1)=K',\quad \Ran(V_2)=K'',\]
we have \[H = H'\overset{\perp}{\oplus}H'' = K'\overset{\perp}{\oplus}K'',\]
and \begin{equation}\label{type3avecred}
    X = \{V_1(w\otimes a_1)U_1^* + V_2(w^\top\otimes a_2)U_2^*,\quad w\in S^p_{I,J}\}.
\end{equation}
By Proposition \ref{propAR}, there exists an element $x\in X, x\geq 0$, such that $\operatorname{s}(x) = 1$. For some $w_0\in S^p_{I,J}$, we have \[x =V_1(w_0\otimes a_1)U_1^* + V_2(w_0^\top\otimes a_2)U_2^*.\] Moreover, since $x$ is injective with dense range, both $w_0\otimes a_1$ and $w_0^\top\otimes a_2$ must be injective with a dense range. Indeed, consider $\xi\in \ell^2_J(H_1)$ such that $(w_0\otimes a_1)(\xi) = 0$. Since $U_1 : \ell^2_J(H_1)\to H$ is an isometry of range $H'$, the map $U_1^*$ is surjective and there exists $\eta\in H'$ such that $\xi = U_1^*(\eta)$. It follows that \[V_1(w_0\otimes a_1)U_1^*(\eta) = 0.\] Since the range of $U_2$ is $H''$, we have $\operatorname{Ker}(U_2^*) = H''^\perp = H'.$ We deduce that \[x(\eta) = V_1(w_0\otimes a_1)U_1^*(\eta) + 0 = 0.\] It follows that $\eta=0$, so $\xi = U_1^*(\eta) = 0$, and we obtain that $w_0\otimes a_1$ is injective. To show that $w_0\otimes a_1$ has a dense range, it suffices to apply the same proof to the adjoint operator $x^* = x = U_1(w_0^*\otimes a_1)V_1^* + U_2((w_0^\top)^*\otimes a_1)V_2^*$. We obtain that $w_0^*\otimes a_1$ is injective, hence $w_0\otimes a_1$ has a dense range. The proof for $w_0^\top\otimes a_2$ injective with a dense range is similar.
It follows that both $w_0$ and $w_0^\top$ are injective with dense range. Since the ranges of $V_1$ and $V_2$ are orthogonal, we have \(V_1^*V_2 = 0\). Then we may compute \begin{equation*}
\begin{split}
    x^*x &= (U_1(w_0^*\otimes a_1)V_1^* + U_2((w_0^\top)^*\otimes a_2)V_2^*)(V_1(w_0\otimes a_1)U_1^* + V_2(w_0^\top\otimes a_2)U_2^*)\\ &= U_1(w_0^*w_0\otimes a_1^2)U_1^* + U_2((w_0^\top)^* w_0^\top\otimes a^2_2)U_2^*.
\end{split}
\end{equation*} Hence, by orthogonality of the ranges of $U_1$ and $U_2$, \[|x| = U_1(|w_0|\otimes a_1)U_1^* + U_2(|w_0^\top|\otimes a_2)U_2^*.\]
Using the identities $x = |x|$, and $U_i^*U_i = 1$ for $i=1,2$, we deduce that \[U_1(|w_0|\otimes a_1) = V_1(w_0\otimes a_1),\quad U_2(|w_0^\top|\otimes a_2) = V_2(w_0^\top\otimes a_2).\]
Then, if we denote by $w_0= O|w_0|$ the polar decomposition of $w_0$, with $O\in B(\ell^2_J,\ell^2_I)$ a unitary operator, we have $w_0^\top = O^\top|w_0^\top|$ (using the same computation as in (\ref{decpoltranspose})). We deduce \[U_1 = V_1(O\otimes 1), \quad U_2 = V_2(O^\top\otimes 1).\] Note that necessarily $|I| = |J|$. 
Finally, we obtain \[X = \left\{U_1(O^*w\otimes a_1)U_1^* + U_2((O^\top)^*w^\top\otimes a_2)U_2^*,\quad w\in S^p_{I,J}\right\}.\]
By making the change of variables $z = O^*w$, and introducing the unitary operator $W_2 := U_2((O^*)^\top\otimes 1) : \ell_J^2(H_2)\to H''$, we obtain \[X = \left\{U_1(z\otimes a_1)U_1^* + W_2(z^\top\otimes a_2)W_2^*,\quad z\in S^p_J\right\}.\]
We can conclude saying that \[X\simp \left\{(z\otimes a_1,z^\top\otimes a_2),\quad z\in S^p_J\right\}.\]

Conversely, let us describe the contractive projection onto \[\left\{(w\otimes a_1,w^\top\otimes a_2),\quad w\in S^p_J\right\}\subset S^p\left(\ell^2_J(H_1)\right)\overset{p}{\oplus}S^p\left(\ell^2_J(H_2)\right) \subset S^p\left(\ell^2_J(H_1)\overset{2}{\oplus}\ell^2_J(H_2)\right),\] where $a_i \in S^p(H_i)$ is a non-zero positive and injective operator on a Hilbert space $H_i$, for $i=1,2$, and $J$ is a countable index set. We may assume that $\|a_1\|^p_p + \|a_2\|^p_p = 1$. We let $\alpha = \|a_1\|^p_p$; it follows that $1-\alpha = \|a_2\|_p^p$.
Arguing as in the second part of the proof of Lemma \ref{decpos1compl}, the contractive projection $$D : S^p\left(\ell^2_J(H_1)\overset{2}{\oplus}\ell^2_J(H_2)\right)\to S^p\left(\ell^2_J(H_1)\overset{2}{\oplus}\ell^2_J(H_2)\right)$$ onto $S^p\left(\ell^2_J(H_1)\right)\overset{p}{\oplus}S^p\left(\ell^2_J(H_2)\right)$ is positive.
For $i=1,2$, let $\varphi_i : S^p(H_i)\to \Cdb$ denote the positive linear form defined by : \[\varphi_i(y) = \dfrac{1}{\|a_i\|_p^p}\Tr(a_i^{p-1}y).\]
Then by Lemma \ref{linear_form_cp}, the map $\nu_i : S^p(\ell^2_J(H_i))\to S^p(\ell^2_J) $ defined by \[\nu_i(x) = (\operatorname{Id}_{S_J^p}\otimes\varphi_i)(x),\]
is a positive map. The map \(\nu : S^p\left(\ell^2_J(H_1)\right)\overset{p}{\oplus}S^p\left(\ell^2_J(H_2)\right)\to S^p\left(\ell^2_J(H_1)\right)\overset{p}{\oplus}S^p\left(\ell^2_J(H_2)\right)\) defined by \[\nu(x,y)= (\nu_1(x)\otimes a_1, \nu_2(y)\otimes a_2)\] is the positive and contractive projection onto the space $[S^p_J\otimes a_1] \overset{p}{\oplus}[S^p_J\otimes a_2]$. Then consider the map $Q : [S^p_J\otimes a_1] \overset{p}{\oplus}[S^p_J\otimes a_2]\to [S^p_J\otimes a_1] \overset{p}{\oplus}[S^p_J\otimes a_2]$ defined by \[Q(z_1\otimes a_1, z_2\otimes a_2) = ((\alpha z_1 + (1-\alpha)z_2^\top)\otimes a_1, (\alpha z_1^\top + (1-\alpha) z_2)\otimes a_2).\] It is straightforward to check that $Q$ is contractive and positive. According to \cite[Section 2]{MRR}, the contractive projection onto $\left\{(w\otimes a_1,w^\top\otimes a_2),\quad w\in S^p_J\right\}$ is $Q\circ \nu \circ D$. Finally, the subspace \[\left\{(w\otimes a_1,w^\top\otimes a_2),~ w\in S^p_J\right\}\] is positively 1-complemented.

Note that in the previous paragraph, if $a_1$ or $a_2$ is zero, the reasoning can be adapted. If $X\simp \{w \otimes a,\quad w\in S^p_J\}$, with $a\in S^p(H_1)$ positive injective and $J$ a countable set, then it is positively 1-complemented.

\subsection{Type 4.} We now consider the case where $X$ is a non-degenerate subspace of type 4 of the space $S^p(H)$, with dimension $N\geq 2$. The goal is to show that there is no positively 1-complemented subspace of this type. Assume that $X$ is positively 1-complemented; Proposition \ref{propAR} will lead to a contradiction. Using definitions and notations (\ref{type4}), for any integer $m\in\{1,\ldots N\}$, there exist a Hilbert space $H_m$, and $a_m\in S^p(H_m)$ some positive and injective operators (or a zero operator), such that $X$ is equivalent to \[Z:=\{(\varphi_1(t)\otimes a_1,\ldots ,\varphi_N(t)\otimes a_N),~t\in \H\} \subset \bigoplus^{p}_{1\leq m\leq N} S^p(\Lambda_N\overset{2}{\otimes} H_m),\] where $\varphi_m : \H \to S^p(\H^{\wedge(m-1)}, \H^{\wedge m})$ is the map defined in (\ref{defvarphi_m}).

We first consider the case where there is only one $a_m$ that is non-zero. Since $a_m$ is injective, the subspace $Z=\{\varphi_m(t)\otimes a_m,\quad t\in\H\}$ of the space $S^p(\H^{\wedge (m-1)}\otimes H_m, \H^{\wedge m}\overset{2}{\otimes} H_m)$ is non-degenerate.
To check that, it is sufficient to prove that \[Y_m:=\{\varphi_m(t),~t\in\H\}\subset S^p(\H^{\wedge (m-1)}, \H^{\wedge m})\] is non-degenerate.
Recall that $Y_m = \text{span}\{x_{k,m},\quad 1\leq k\leq N\}$ (see (\ref{defvarphi_m})), and let us show that \begin{equation}\label{sl4}
    \text{span}\left\{\bigcup_{1\leq k\leq N}\Ran(x_{k,m})\right\}=\H^{\wedge m}.
\end{equation}
Let $(e_k)_{1\leq k\leq N}$ be the canonical basis of $\H$, then the family \[\{e_{i_1}\wedge\cdots \wedge e_{i_m},\quad 1\leq i_1 < i_2 < \cdots  < i_m \leq N\}\] is an orthonormal basis of $\H^{\wedge m}$. For any $v = e_{i_1}\wedge\cdots \wedge e_{i_m} $ in this family, we have \[v = x_{i_1,m}(e_{i_2}\wedge\cdots \wedge e_{i_m}) \in \Ran(x_{i_1,m}).\]
The equality (\ref{sl4}) follows, and we deduce that \[\text{s}_\ell(Y_m) = 1.\]
For the right support, we can use a similar argument. 
Let $u = e_{i_1}\wedge\cdots \wedge e_{i_{m-1}}$ be a vector in the standard orthonormal basis of $\H^{\wedge (m-1)}$. Then for $1\leq i_m\leq N$ such that $i_m \notin \{i_1,\ldots , i_{m-1}\}$, \[u = x_{i_m, m}^*(e_{i_m}\wedge u)\in \Ran(x^*_{i_m,m}) = \operatorname{Ker}(x_{i_m,m})^\perp.\]
Consequently, \[\text{span}\left\{\bigcup_{1\leq k\leq N} Ker(x_{k,m})^\perp\right\} = \H^{\wedge(m-1)},\] and hence \[\text{s}_r(Y_m) = 1.\]
By applying point (2) of Remark \ref{rem_before_proof}, we deduce that there exist unitary operators \[U_m : \H^{\wedge (m-1)}\overset{2}{\otimes} H_m\to H \text{ and } V : \H^{\wedge m}\overset{2}{\otimes} H_m\to H,\] such that \[X = \{V_m(\varphi_m(t)\otimes a_m)U_m^*,\quad t\in \H\}.\]
By Proposition \ref{propAR}, there exists a positive element $x\in X$, such that $\operatorname{s}(x) = 1$. This element can be written as $x = V_m(\varphi_m(t_0)\otimes a_m)U_m^*$, with some $t_0\in \H.$ It implies that $\varphi_m(t_0)\otimes a_m$ is injective with dense range. The proof of this fact is similar to the one appearing in Type 4.  Since $\H^{\wedge m}$ is finite-dimensional, we deduce that $\varphi_m(t_0)\in Y_m\subset S^p(\H^{\wedge (m-1)}, \H^{\wedge m})$ is a bijective map.

To reach a contradiction, we prove that there is no bijective element in $Y_m$. Assume that $y_m$ is a bijective element of $Y_m$. Then it can be expressed as \[y_m := \sum_{k=1}^N t_k x_{k,m} : \H^{\wedge (m-1)}\to \H^{\wedge m},\quad t =(t_k)_k\in\H.\]
Necessarily, \(\text{dim}(\H^{\wedge(m-1)}) = \text{dim}(\H^{\wedge m}),\) hence $N$ must be odd and $m=\frac{N+1}{2}.$
Define $z_m = y_m^*y_m$. Then $z_m$ is a bijection. This yields a contradiction since the rank of $z_m$ is equal to $\binom{N-1}{m-1} < \binom{N}{m-1} = \text{dim}(\H^{\wedge (m-1)})$. To check this, we follow the computations from the proof of \cite[Proposition 2.4, p.18]{AF2}. Define $y:= \sum_{k=1}^N t_k c_k \in B(\La_N)$. By relations (\ref{CAR}), we have \[y^*y+yy^* = \|t\|^2_2 \operatorname{Id}_{\La_N} \text{ and } yy^*y = \|t\|_2^2y.\] 
Restricting these equalities to $\H^{\wedge(m-1)}$ for $1\leq m\leq N$, we obtain \begin{equation}\label{tracey_m}
    y_1^*y_1 = \|t\|^2_2,~ y_m^*y_{m-1}+y_{m-1}y_m^* = \|t\|_2^2 \operatorname{Id}_{\mathcal{H}^{\wedge( m-1)}},~y_Ny_N^* = \|t\|_2^2 \operatorname{Id}_{\mathcal{H}^{\wedge N}}, ~~1<m<N,
\end{equation}
and \begin{equation}\label{diagozm}
    y_my_m^*y_m=\|t\|^2_2 y_m,\quad 1\leq m\leq N.
\end{equation}
By induction using (\ref{tracey_m}), we obtain \[\Tr(z_m)=\Tr(y_m^*y_m)= \|t\|_2^2 \binom{N-1}{m-1},\quad 1\leq m\leq N,\] and by (\ref{diagozm}), \[z_m^2 = \|t\|_2^2 z_m,\quad 1\leq m\leq N.\]
We deduce that the spectrum of $z_m$ consists of $\|t\|_2^2$ with multiplicity $\binom{N-1}{m-1}$ and of 0, which yields the value of the rank of $z_m$.

We now turn to the case of several blocks.
Assume that the operators $a_i$ are not all zero, and denote by $a_{i_1}$,\ldots , $a_{i_k}$ those that are non-zero. Then $X$ is equivalent to the space $Z$, defined as \[Z= \{(\varphi_{i_1}(t)\otimes a_{i_1},\ldots ,\varphi_{i_k}(t)\otimes a_{i_k}),~t\in \H\} \subset \bigoplus^{p}_{1\leq j\leq k} S^p(\H^{\wedge (i_j -1)}\overset{2}{\otimes} H_{i_j}, \H^{\wedge i_j}\overset{2}{\otimes} H_{i_j}).\]
By the first part of the proof, each \(\{\varphi_{i_j}(t)\otimes a_{i_j},~t\in\H \}\) is non-degenerate. This readily implies that $Z$ is non-degenerate.
Now, using a reasoning similar to the one at the beginning of Subsection 4.3, we obtain two orthogonal decompositions of $H$ : \[H = H_{i_1}'\overset{\perp}{\oplus}\cdots \overset{\perp}{\oplus}H'_{i_k} = K_{i_1}'\overset{\perp}{\oplus}\cdots \overset{\perp}{\oplus}K'_{i_k},\]
and unitaries \[U_{i_j} : \H^{\wedge (i_j -1)}\overset{2}{\otimes} H_{i_j}\to H'_{i_j} \quad \text{and}\quad V_{i_j} : \H^{\wedge i_j}\overset{2}{\otimes} H_{i_j} \to K'_{i_j},\quad 1\leq j\leq k,\]
such that \[X = \left\{\sum_{j=1}^k V_{i_j}(\varphi_{i_j}(t)\otimes a_{i_j})U_{i_j}^*,\quad t\in \H\right\}.\]
Now, using Proposition \ref{propAR}, we have a positive element $x\in X$ such that $\operatorname{s}(x) =1$. We may write this element as \[x = \sum_{j=1}^k V_{i_j}(\varphi_{i_j}(t_0)\otimes a_{i_j})U_{i_j}^*,\]
 with $t_0\in \H$. Since $x$ is positive and injective, and since the operators $U_{i_j}$ and $V_{i_j}$ are unitaries, we deduce that the operator $\varphi_{i_j}(t_0)\otimes a_{i_j}$ is injective with dense range, for each $j$. It implies that $\varphi_{i_j}(t_0)$ is bijective, leading to a contradiction, as seen before. We conclude that there is no positively 1-complemented subspace of type 4.

\subsection{Type 5.} We now turn to the spinorial type, of even dimension. Let $X$ be a non-degenerate spinorial space of dimension $2N\geq 4$ (see Definition \ref{type56}). Assume that it is positively 1-complemented. 
Since $1\in \F_N^p$, the subspace \(\F_N^p\) of \(S^p(\La_N)\) is non-degenerate. Recall that the notation $\F_N^p$ refers to the space $\F_N \subset B(\La_N)$ equipped with the norm of $S^p(\La_N)$. Arguing as in type 3, we can describe $X$ as follows.
The space $H$ can be decomposed as \[H = H'\overset{\perp}{\oplus}H'' = K'\overset{\perp}{\oplus}K'',\]
there exist two positive and injective operators on Hilbert spaces, denoted by $a_1\in S^p(H_1)$ and $a_2\in S^p(H_2)$ that verify \[\|a_1\|_p^p + \|a_2\|_p^p =1,\] and unitaries \[U_1 : \La_N\overset{2}{\otimes} H_1 \to H',~V_1 : \La_N\overset{2}{\otimes} H_1 \to K',~U_2 : \La_N\overset{2}{\otimes} H_2\to H'',~V_2 : \La_N\overset{2}{\otimes} H_2\to K''\]
such that \[X = \left\{V_1(z\otimes a_1)U_1^* + V_2(\sigma(z)\otimes a_2)U_2^*,\quad z\in \F_N^p\right\}.\]
The definition of the map $\sigma : \F_N \to \F_N$ is given in (\ref{defsigma}).
In the sequel, we assume that $a_1$ and $a_2$ are non-zero. If one of them is zero, all that follows adapts accordingly.
By Proposition \ref{propAR}, there is a positive element $x\in X$, such that $\operatorname{s}(x)=1$. This element can be written as \[x = V_1(z_0\otimes a_1)U_1^* + V_2(\sigma(z_0)\otimes a_2)U_2^*,\]
where $z_0\in\F^p_N.$
As in type 3, the injectivity of $x$ implies that the elements $z_0$ and $\sigma(z_0)$ are injective, so they are bijective elements of $\F_N$.
Using the identity \(x = |x|\), we obtain \begin{equation}\label{etapecalculcas5}
    V_1(z_0\otimes a_1) = U_1(|z_0|\otimes a_1)\quad \text{ and }\quad V_2(\sigma(z_0)\otimes a_2) = U_2(|\sigma(z_0)|\otimes a_2).
\end{equation}
We can deduce that $U_1$ and $V_1$ have the same range, that is, $H'=K'$. Similarly, $H''=K''$.
Let $z_0 = u|z_0|$ and $\sigma(z_0) = w|\sigma(z_0)|$, be the polar decompositions of $z_0$ and $\sigma(z_0)$ respectively, where $u,w\in B(\La_N)$ are unitary operators. Recall from Section \ref{sectionJC*triple} that the space $\F_N$ is a $JC^*$-triple and note that the map $\sigma$ is a triple-homomorphism on $\F_N$.  This implies that the space \[\{(x,\sigma(x)),~x\in\F_N\}\subset B(\La_N)\overset{\infty}{\oplus}B(\La_N)\] is a $JC^*$-triple. Consider the bijective element $(z_0,\sigma(z_0))$ with its polar decomposition \( (z_0,\sigma(z_0)) = (u,w)(|z_0|,|\sigma(z_0)|)\). We may apply Proposition \ref{propfunctiononJC} on it, with a continuous function equal to 1 on the spectrum of this element and equal to 0 at 0. We obtain \[(u,w)\in \{(x,\sigma(x)),~x\in\F_N\},\] i.e. $w=\sigma(u)$. Now, using identities (\ref{etapecalculcas5}) again, we obtain \[V_1^*U_1 = u\otimes 1,\quad \text{ and }\quad V_2^*U_2 = \sigma(u)\otimes 1.\]
Finally, we can give the following description for $X$,
\begin{equation*}
    \begin{split}
        X &= \left\{U_1(u^*x\otimes a_1)U_1^*  + U_2(\sigma(u)^*\sigma(x)\otimes a_2)U_2^*,~x\in\F_N^p\right\},
    \end{split}
\end{equation*}
Note that $\F_N$ is a unital $JC^*$-triple, and $\sigma$ is a unital triple-homomorphism, so we may apply Lemma \ref{JCautoadj}. Since $u\in \F_N$, we have $u^*\in\F_N$, and $\sigma$ is an adjoint preserving map. By setting $v=u^*$, we conclude that there is a unitary element $v\in \F_N$ such that \[X\simp \left\{(vx\otimes a_1, \sigma(v)\sigma(x)\otimes a_2),\quad x\in \F_N^p \right\}.\]

Conversely, we describe the contractive projection onto \[Y_v :=\left\{(vx\otimes a_1, \sigma(v)\sigma(x)\otimes a_2),\quad x\in \F_N^p \right\}\subset S^p\left((\La_N\overset{2}{\otimes} H_1)\overset{2}{\oplus}(\La_N\overset{2}{\otimes} H_2)\right),\] where $v\in\F_N$ is a unitary operator and $a_i$ is a positive operator on a Hilbert space $H_i, ~i=1,2$. Let us show that this projection is positive. We may assume, without loss of generality, that $\|a_1\|_p^p + \|a_2\|_p^p = 1$. We also assume that both $a_1$ and $a_2$ are non-zero, otherwise the following argument adapts readily.
Similarly to the converse case of type 3, the contractive projection $D$ from $S^p\left((\La_N\overset{2}{\otimes} H_1)\overset{2}{\oplus}(\La_N\overset{2}{\otimes} H_2)\right)$ onto $S^p(\La_N\overset{2}{\otimes} H_1)\overset{p}{\oplus} S^p(\La_N\overset{2}{\otimes} H_2)$ is positive. Then, using the same notations as in the converse case of type 3, we consider the positive linear functionals $\varphi_i : S^p(H_i)\to \Cdb$, and the positive maps $\nu_i = \operatorname{Id}_{S^p(\La_N)}\otimes \varphi_i : S^p(\La_N\overset{2}{\otimes} H_i)\to S^p(\La_N)$. The map $\nu : S^p(\La_N\overset{2}{\otimes} H_1)\overset{p}{\oplus} S^p(\La_N\overset{2}{\otimes} H_2)\to S^p(\La_N\overset{2}{\otimes} H_1)\overset{p}{\oplus} S^p(\La_N\overset{2}{\otimes} H_2)$ defined by \[\nu(x,y) = (\nu_1(x)\otimes a_1, \nu_2(y)\otimes a_2)\] is the positive and contractive projection onto the space $S^p(\La_N)\otimes a_1 \overset{p}{\oplus} S^p(\La_N)\otimes a_2$.
Denote by $Q : S^p(\La_N)\to S^p(\La_N)$ the contractive projection onto $\F_N^p$. By \cite[Lemma 5.3, p.867]{MRR}, it is a unital and positive projection. Note that this lemma shows that for every $1\leq q\leq \infty$, \[\|Q : S^q(\La_N)\to S^q(\La_N)\|=1.\] We use it with $q=p$ and $q=\infty$. The proof of \cite[Proposition 5.7, p.869]{MRR} provides the expression of the projection $R$ from $[S^p(\La_N)\otimes a_1]\overset{p}{\oplus} S^p(\La_N)\otimes a_2]$ onto the space \(\{(x\otimes a_1, \sigma(x)\otimes a_2),\quad x\in\F_N^p\}\). Denote \(\alpha = \|a_1\|^p_p\). Then \[R(z_1\otimes a_1, z_2\otimes a_2) = \left((\alpha Q(z_1) + (1-\alpha) \sigma(Q(z_2)))\otimes a_1, (\alpha \sigma(Q(z_1)) + (1- \alpha) Q(z_2) )\otimes a_2 \right).\]
Now, using the identity \[Y_v = \left\{(vx\otimes a_1, \sigma(v)\sigma(x)\otimes a_2),\quad x\in \F_N^p \right\} = (v\otimes 1,\sigma(v)\otimes 1)\{(x\otimes a_1, \sigma(x)\otimes a_2),\quad x\in\F_N^p\},\]
where $(v\otimes 1,\sigma(v)\otimes 1) \in B((\La_N\overset{2}{\otimes} H_1)\overset{2}{\oplus}(\La_N\overset{2}{\otimes} H_2))$ is a unitary operator, we can write the expression of the contractive projection $R_v$ onto \(Y_v \subset S^p((\La_N\overset{2}{\otimes} H_1)\overset{2}{\oplus}(\La_N\overset{2}{\otimes} H_2))\) : 
\begin{equation*}
    \begin{split}
        R_v(z_1\otimes a_1, z_2\otimes a_2) &= (v\otimes 1,\sigma(v)\otimes 1)R(v^*z_1\otimes a_1, \sigma(v)^*z_2\otimes a_2)\\
        &= (v(\alpha Q(v^*z_1) + (1-\alpha) \sigma(Q(\sigma(v)^*z_2)))\otimes a_1 ,\\&\hspace{3cm} \sigma(v)(\alpha \sigma(Q(v^*z_1)) + (1-\alpha) Q(\sigma(v)^*z_2))\otimes a_2).
    \end{split}
\end{equation*}
To see why $R_v$ is a positive map, we use that the element \((z_1\otimes a_1, z_2\otimes a_2)\) is positive if, and only if $(z_1,z_2)$ is positive. Then it suffices to show that the map \(\Tilde{R}_v : B(\La_N)\overset{\infty}{\oplus}B(\La_N)  \rightarrow  B(\La_N)\overset{\infty}{\oplus}B(\La_N)\) defined as \[
    \Tilde{R}_v(z_1,z_2) = \left(v(\alpha Q(v^*z_1) + (1-\alpha) \sigma(Q(\sigma(v)^*z_2))), \sigma(v)(\alpha \sigma(Q(v^*z_1)) + (1-\alpha) Q(\sigma(v)^*z_2)) \right)\] is positive. Note the map $Q : B(\La_N)\to B(\La_N)$ is contractive and $\sigma : B(\La_N)\to B(\La_N)$ is isometric. It implies that $\Tilde{R}_v$ is unital and contractive, and therefore positive. To check that $\Tilde{R}_v$ is unital, the two key points are: (1) $v\in \F_N$, so $v^*\in\F_N$, and $Q(v^*) = v^*$; (2) $\sigma(\F_N)\subset \F_N,$ and $\sigma$ is adjoint preserving.
The contractive projection onto \(Y_v\) is the composition of positive and contractive projections $D$, $\nu$ and $R_v$, so this subspace is positively 1-complemented.

\subsection{Type 6.} Consider the last type, which is very similar to the previous one. Let $X\subset S^p(H)$ be a non-degenerate spinorial space of odd dimension $2N+1\geq 5$. Assume that it is positively 1-complemented. Note that $1\in \E_{2N}^p$, so the subspace $\E_{2N}^p$ of $S^p(\La_N)$  is non-degenerate. We can describe the space $X$ as 
\[X = \{V(z\otimes a)U^*,\quad z\in\E_{2N}^p\},\]
where $a\in S^p(H_1)$ is a positive and injective operator on a Hilbert space $H_1$, and $U,V : \La_N\overset{2}{\otimes} H_1\to H$ are unitary operators.
Using Proposition \ref{propAR}, we have a positive element $x\in X$, such that $\operatorname{s}(x)=1$, that can be written as \(x = V(z_0\otimes a)U^*\), with \(z_0\in\E_{2N}^p\) a bijective operator.
We denote by \(z_0 = u|z_0|\) its polar decomposition, with $u\in B(\La_N)$ a unitary operator. Since $\E_{2N}$ is a $JC^*$-triple, and $z$ is bijective, we obtain that \(u\in\E_{2N}\), by Proposition \ref{propfunctiononJC} as in the type 5 case.
Now, using the identity \(x_0 = |x_0|\), we obtain the following equality \[U(|z_0|\otimes a)U^* = V(z_0\otimes a)U^* = V(u|z_0|\otimes a)U^*,\] from which we deduce \[V(u\otimes 1) = U.\]
Consequently, \[X = \{U(u^*y\otimes a)U^*,\quad y\in \E_{2N}^p\}\simp \{x\otimes a,\quad x\in u^*\E_{2N}^p\}.\] Moreover, $v=u^*\in \E_{2N}$.

Conversely, assume that $X$ is positively equivalent to \(\{x\otimes a,\quad x\in v\E_{2N}^p\}\), with $v\in\E_{2N}$ a unitary operator, and $a\in S^p(H_1)$ a positive and injective operator. We may assume that $\|a\|_p=1$. We write the expression of the projection onto this subspace.
Denote by $R : S^p(\La_N)\to S^p(\La_N)$ the contractive projection onto $\E_{2N}^p$. By \cite[Lemma  5.2, p.865]{MRR}, this map has the following property \[\|R : S^q(\La_N)\to S^q(\La_N)\| = 1,\] for any $1\leq q\leq \infty$. We will apply it with $q=p$ and $q=\infty$. 
Then, denote $R_v : S^p(\La_N)\to S^p(\La_N)$ the contractive projection onto $v\E_{2N}^p$. Then $R_v$ is given by \[R_v(x) = vR(v^*x).\] Since $v$ is a unitary operator of $\E_{2N}$, we also have \[\|R_v : B(\La_N)\to B(\La_N)\| = 1,\] and \(R_v(1) = vR(v^*) = vv^* =1.\)
We deduce that $R_v$ is a positive map. Then, define $\varphi_a : S^p(H_1)\to \Cdb$ as follows \[\varphi_a(x) = \Tr(a^{p-1}x).\] The contractive projection onto \(\{x\otimes a,\quad x\in v\E_{2N}^p\}\subset S^p(\La_N\overset{2}{\otimes}H_1)\) is the map \[ P: x \mapsto (R_v\otimes \varphi_a)(x)\otimes a.\]
Since the maps $\varphi_a$ and $R_v$ are positive and contractive, it follows from Remark \ref{Totimesphi} (2) that $R_v\otimes\varphi_a$ is a positive map, and therefore $P$ is positive as well.

\subsection{The case $p=1$.}

For the case $p=1$, the difference comes from the fact that for a 1-complemented subspace $X$ of $S^1(H)$, the contractive projection onto $X$ is not necessarly unique. However, according to \cite[the remark after Theorem 2.16, p.36]{AF2}, there is a unique contractive projection $P$ onto $X$ satisfying \begin{equation}\label{unique_proj_S1}
    P = P\left(\text{s}_\ell(X) \cdot \text{s}_r(X)\right).
\end{equation} Hence, if $X$ is indecomposable and non-degenerate, there exists a unique contractive projection $P : S^1(H)\to S^1(H)$ onto $X$. It follows that the previous study of each type remains valid in the case $p=1$, and thus Theorem \ref{THM1} holds for $p=1$.



\begin{remark}
A subspace $X$ of the space $S^1(H)$ may be positively 1-complemented even though not all contractive projections onto it are positive. For example, let \(a = \left(\begin{array}{ccc}
     1& 0&0 \\
     0& 0&0\\
     0&0&0
\end{array}\right) \in S^1(\ell^2_3) = S^1_3,\) and \(X:= \text{span}(a).\) Let $v = \left(\begin{array}{ccc}
     1& 0&0 \\
     0&0&0\\
     0&0&0
\end{array}\right)$, then the map $ Q: S^1_3\to S^1_3 $ defined by \[Q(x) = \Tr(xv)a\] is a contractive and positive projection onto $X$.
However, let \[w = \left(\begin{array}{ccc}
     1&0&0 \\
     0& 0&0\\
     0&1&0
\end{array}\right),\] the map $P_{w} : S^1_3\to S^1_3 $ defined by \[P_{w}(x) = \Tr(xw)a\] is also a contractive projection onto $X$, but not positive.
\end{remark}

\section{Proof of Theorem \ref{THM1}, for $p=2$.}\label{section}

In this section, we prove Theorem \ref{THM1} in the case $p=2$. Note that only the direct implication needs to be established, as the converse was already addressed in Section \ref{sectionProofpneq2} and remains valid for $p=2$. Throughout this section, $X$ denotes a non-degenerate, positively 1-complemented, and indecomposable subspace of $S^2(H)$. We denote by $P : S^2(H)\to S^2(H)$ the contractive positive projection onto $X$.
The strategy to obtain the desired result is to associate to $P$ certain positive and contractive projections $V_p$ on $S^p(H)$ for $p\neq 2$, then apply the case $p\neq 2$ of Theorem \ref{THM1} to determine their ranges, and finally return to the subspace $X$. The key point is to show that the map $V_p$ is contractive, this relies on the case $p=1$ and an interpolation argument.
\subsection{Construction of positive contractive projection on $S^p(H)$, $p\neq 2$}
Some preliminary results, stated in \cite[Lemma 2.3, p.10]{AR}, are reviewed here. We provide the proof of the second point of Lemma \ref{densite}, as it illustrates how to approach an element of $\text{s}(k)S^p(H)\text{s}(k)$ by elements of $k^{\frac{1}{2}}B(H)k^{\frac{1}{2}}$ in a concrete way.

\begin{lemma}\label{densite}
Let $1\leq p <\infty$, and let $k$ be a positive element of $S^p(H)$.

\begin{itemize}
    \item The map $\operatorname{s}(k)B(H)\operatorname{s}(k)\to S^p(H)$, $x \mapsto k^{\frac{1}{2}}xk^{\frac{1}{2}}$ is injective.
    \item The subspace $k^{\frac{1}{2}}B(H)k^{\frac{1}{2}}$ is dense in $\operatorname{s}(k)S^p(H)\operatorname{s}(k)$ for the topology of $S^p(H)$.
\end{itemize}
\end{lemma}

\begin{proof} By convenience, we assume that $k$ has an infinite rank (the finite rank case being easier).
The compact operator $k$ is adjoint preserving, so we can find an orthonormal basis $(e_n)_{n\in\mathbb N}$ of the space $\text{Ker}(k)^\perp = \overline{\text{Ran}(k)}$ consisting of eigenvectors of $k$. For every $n\in\Ndb$, let $\lambda_n\in\Rdb_+^*$ such that \(ke_n = \lambda_n e_n\). We may assume that the sequence $(\lambda_n)_n$ is non-increasing. Let $E_{nn}$ be the operator defined as \[E_{nn}(e_m) = \delta_{n,m}e_n,~~n,m\in\Ndb \quad \text{ and }\quad {E_{n,n}}_{|_{\text{Ker}(k)}} = 0,\]
then we can write $k$ as \[k = \sum_{n\in\mathbb N} \lambda_n E_{nn}.\] We denote by $p_N$ the orthogonal projection onto $\text{span}\{e_1,\ldots ,e_N\}$, and we define two operators as follows, \[k_N :=p_Nkp_N = \sum_{1\leq n\leq N} \lambda_n E_{nn}, \quad \text{ and }\quad g_N :=\sum_{1\leq n\leq N} \lambda_n^{-1} E_{nn}.\]
For any element $y\in \text{s}(k)S^p(H)\text{s}(k)$, we have \[p_Nyp_N\xrightarrow[N\to +\infty]{S^p}y,\]
and it is straightforward to check that \(k^{\frac{1}{2}}g_N^{\frac{1}{2}} = g_N^{\frac{1}{2}}k^{\frac{1}{2}} = p_N.\) Then we have \(p_Nyp_N = k^{\frac{1}{2}}\underbrace{g_N^{\frac{1}{2}} yg_N^{\frac{1}{2}}}_{\in B(H)}k^{\frac{1}{2}} \in k^{\frac{1}{2}}B(H)k^{\frac{1}{2}}.\) The result follows.

\end{proof}

From now on, we fix $h$ a positive and injective element of $X$, with $\|h\|_2 =1$. The existence of $h$ is ensured by Proposition \ref{propAR}. Note that here, $s(h)=1.$ We define the operator $V_1 : hB(H)h\to S^1(H)$ by \begin{equation}\label{defV1}
    V_1(hxh) = h^{\frac{1}{2}}P(h^{\frac{1}{2}}xh^{\frac{1}{2}})h^{\frac{1}{2}}.
\end{equation}
We see that the map $V_1$ is well defined, using Lemma \ref{densite} with $k=h^2 \in S^1(H)$. Indeed, we have $s(k)=s(h)=1$, and for $x,y\in B(H)$, we have \[hxh=hyh\underset{\text{Lemma }\ref{densite} }{\Rightarrow} x=y \Rightarrow h^{\frac{1}{2}}P(h^{\frac{1}{2}}xh^{\frac{1}{2}})h^{\frac{1}{2}} =h^{\frac{1}{2}}P(h^{\frac{1}{2}}yh^{\frac{1}{2}})h^{\frac{1}{2}}.\] Note that the implication \(hxh = hyh \Rightarrow x=y\) may be obtained directly. Since $h$ is injective, \(hxh=hyh\Rightarrow xh=yh,\) and $h$ has a dense range, hence \(xh=yh \Rightarrow x=y.\)
As in the previous lemma, we diagonalize the operator $h$, by writing it as \[h = \sum_{n\in\mathbb N}\alpha_n E_{nn},\]
with $(e_n)_n$ an orthonormal basis of $H$, and $(\alpha_n)_n\subset \Rdb^*_+$, such that $he_n = \alpha_n e_n.$ We denote by $p_N$ the orthogonal projection onto $\text{Span}\{e_1,\ldots ,e_N\}$. Then the proof of Lemma \ref{densite} shows that the subspace $\bigcup\limits_{N\in\mathbb N} p_NB(H)p_N$ of $hB(H)h$ is dense in $S^1(H)$.
We now use this result to extend $V_1$ to a bounded operator on $S^1(H)$. 

\begin{proposition}
The operator $V_1$ defined in (\ref{defV1}) extends to a contractive and positive projection from $S^1(H)$ into itself.
\end{proposition}
\begin{proof} We first show that $V_1$ is bounded on a dense subspace of the space $S^1(H)$ to obtain a bounded extension of $V_1$ on $S^1(H)$. The map $V_1$ is a positive projection since $P$ is. Then, to show that $V_1$ is a contraction, we show that its adjoint $V_1^* : B(H)\to B(H)$ is unital.

We assume that $h$ has an infinite rank. We first show that $V_1$ is bounded on $p_NB(H)p_N \subset S^1(H)$, uniformly in $N\in\Ndb$.
Note that if $x$ is in $B(H)_+$, the set of all positive elements of $B(H)$, then $V_1(hxh) =  h^{\frac{1}{2}}P(h^{\frac{1}{2}}xh^{\frac{1}{2}})h^{\frac{1}{2}}$ is positive, since the map $P$ is assumed to be positive. In this case, we can compute \begin{equation*}\begin{split}
    \|V_1(hxh)\|_1 &= \| h^{\frac{1}{2}}P(h^{\frac{1}{2}}xh^{\frac{1}{2}})h^{\frac{1}{2}}\|_1\\ &= \Tr( h^{\frac{1}{2}}P(h^{\frac{1}{2}}xh^{\frac{1}{2}})h^{\frac{1}{2}})\\ &= \Tr( hP(h^{\frac{1}{2}}xh^{\frac{1}{2}}))\\ &= \Tr(P(h)h^{\frac{1}{2}}xh^{\frac{1}{2}})\\ &= \Tr(hh^{\frac{1}{2}}xh^{\frac{1}{2}}) = \Tr(hxh) = \|hxh\|_1.
\end{split}
\end{equation*}
The crucial points here are the equality $P(h)=h$, and the fact that $P : S^2(H)\to S^2(H)$ is an orthogonal, hence self-adjoint and positive projection. Any operator $y \in p_NB(H)p_N$ may be decomposed as a combination of positive operators \[y = y_1 - y_2 + i(y_3-y_4),\] with $y_i \in p_NB(H)p_N$, $y_i\geq 0$ and $\|y_i\|_1\leq \|y\|$, for $i=1,\ldots ,4$.
As in the proof of Lemma \ref{densite}, we have \begin{equation}\label{etapecalc1}
    y_i = p_Ny_ip_N = h\underbrace{g_Ny_ig_N}_{:=x_i\in B(H)_+}h,
\end{equation} with $g_N = \sum\limits_{n=1}^N \alpha_n^{-1}E_{nn}$. We obtain \[\|V_1(y)\|_1 \leq \sum_{i=1}^4 \|V_1(y_i)\|_1 = \sum_{i=1}^4 \|y_i\|_1 \leq 4\|y\|_1.\]
This shows that $V_1$ is bounded on $\bigcup\limits_{N\in\mathbb N} p_NB(H)p_N $, so it extends to a bounded map on $S^1(H)$, with norm less than 4.
We can now show that $V_1$ is a positive projection. If $y\in S^1(H)_+$, then arguing as in (\ref{etapecalc1}), we see that $p_Nyp_N\subset hB(H)_+h$ for all $N\geq 1$. Then the sequence of positive elements $(V_1(p_Nyp_N))_N$ converges to $V_1(y)$, which is therefore positive.
To check the equality \(V_1^2 = V_1\), we only need to prove it on the subspace $hB(H)h$. For $x\in B(H)$, let \[y = V_1(hxh) = h^{\frac{1}{2}}P(h^{\frac{1}{2}}xh^{\frac{1}{2}})h^{\frac{1}{2}}.\]
On the one hand, $V_1(p_Nyp_N)\xrightarrow[N\to \infty]{S^1}V_1(y)$, on the other hand, we can compute \begin{equation*}
    \begin{split}
        V_1(p_Nyp_N) &=  h^{\frac{1}{2}}P(h^{\frac{1}{2}}g_Nyg_Nh^{\frac{1}{2}})h^{\frac{1}{2}} \\ &=h^{\frac{1}{2}}P(h^{\frac{1}{2}}g_Nh^{\frac{1}{2}}P(h^{\frac{1}{2}}xh^{\frac{1}{2}})h^{\frac{1}{2}}g_Nh^{\frac{1}{2}})h^{\frac{1}{2}} \\ & = h^{\frac{1}{2}}P(p_NP(h^{\frac{1}{2}}xh^{\frac{1}{2}})p_N)h^{\frac{1}{2}} \xrightarrow[N\to\infty]{S^1} h^{\frac{1}{2}}P(P(h^{\frac{1}{2}}xh^{\frac{1}{2}}))h^{\frac{1}{2}} = h^{\frac{1}{2}}P(h^{\frac{1}{2}}xh^{\frac{1}{2}})h^{\frac{1}{2}}.
    \end{split}
\end{equation*}
We deduce the equality \(V_1(y)=y\), and hence \(V_1^2 = V_1 \text{ on }hB(H)h,\) as expected.

To show that $V_1$ is a contraction, we use duality brackets between $S^1(H)$ and $B(H)$, denoted by $\langle \cdot , \cdot \rangle$. Let $U = V_1^* : B(H)\to B(H)$. This is a positive bounded map, hence \(\|V_1\| = \|U\| = \|U(1)\|\) by Russo-Dye Theorem, see \cite[Corollary 2.9, p.15]{Paulsen}. For every $x\in B(H)$, using again the identities \(P(h)=h\) and the fact that $P$ is self-adjoint, we have \begin{equation*}\begin{split}
    \langle U(1), hxh\rangle &= \langle 1, V_1(hxh)\rangle\\ &= \langle 1, h^{\frac{1}{2}}P(h^{\frac{1}{2}}xh^{\frac{1}{2}})h^{\frac{1}{2}}\rangle\\ &= \Tr(h^{\frac{1}{2}}P(h^{\frac{1}{2}}xh^{\frac{1}{2}})h^{\frac{1}{2}})\\ &= \Tr(hP(h^{\frac{1}{2}}xh^{\frac{1}{2}})))\\ &= \Tr(P(h)h^{\frac{1}{2}}xh^{\frac{1}{2}}) = \Tr(hxh) = \langle 1, hxh\rangle.
\end{split}
\end{equation*}
We conclude that \(U(1) = 1\), and hence that \(\|V_1\| = 1\). 

Note that if $h$ has finite rank (that is, $H$ is of finite dimension), we may directly define \[V_1(x) = h^{-\frac{1}{2}}P(h^{-\frac{1}{2}}xh^{\frac{1}{2}})h^{\frac{1}{2}}.\] Then to prove that $V_1$ is a contractive and positive projection on $S^1(H)$ we proceed in the same way, without approximation of $h^{-1}$.
\end{proof}
Using Lemma \ref{densite} again, we now define $V_p : h^{\frac{1}{p}}B(H)h^{\frac{1}{p}}\to S^p(H)$ for $1\leq p\leq 2$, by \begin{equation}\label{defVp}
    V_p(h^{\frac{1}{p}}xh^{\frac{1}{p}}) = h^{\frac{1}{p} - \frac{1}{2}}P(h^{ \frac{1}{2}}xh^{ \frac{1}{2}})h^{\frac{1}{p} - \frac{1}{2}}.
\end{equation}

\begin{proposition}
The operator $V_p$ defined in (\ref{defVp}), for $1\leq p\leq 2$, extends to a positive contractive projection from $S^p(H)$ into itself.
\end{proposition}

\begin{proof}

To check that $V_p$ extends to a contraction on $S^p(H)$, we use Stein's interpolation theorem, refer e.g. to \cite{Stein}. It suffices to show that, for any $N,m\in\Ndb$, the map $T_{p,m,N} : (p_NB(H)p_N, \|.\|_p) \to  (p_mB(H)p_m, \|.\|_p)$ defined by \[T_{p,m,N}(x) = p_mV_p(x)p_m = h_m^{\frac{1}{p}-\frac{1}{2}}P(g_N^{\frac{1}{p}-\frac{1}{2}}xg_N^{\frac{1}{p}-\frac{1}{2}})h_m^{\frac{1}{p}-\frac{1}{2}}\] is a contraction. 

Fix $N,m\in \Ndb$, and let $H_N := p_N(H)$, $H_m := p_m(H)$. We have natural identifications \(B(H_N) \simeq p_NB(H)p_N\) and $B(H_m) \simeq p_mB(H)p_m$. Define the strip \(S = \{z\in\Cdb,~0\leq Re(z) \leq 1 \}\) and for any $z\in S$, define the map $T(z) : B(H_N)\to B(H_m)$ by \[[T(z)](x) = h_m^{\frac{z}{2}}P(g_N^{\frac{z}{2}} x g_N^{\frac{z}{2}}) h_m^{\frac{z}{2}}.\] The maps $z\mapsto h_m^{\frac{z}{2}}\in B(H_m)$ and $z\mapsto g_N^{\frac{z}{2}}\in B(H_N)$ are analytic on $\Cdb$. It implies the two following facts: \begin{enumerate}
    \item For all $x\in B(H_N)$, the map $z\mapsto [T(z)](x)$ is continuous on $S$ and analytic on $\overset{\circ}{S}$.
    \item For all $x\in B(H_N)$, the maps \[s\in\Rdb\mapsto [T(is)](x) \in S^2(H_m) \text{ and  } s\in\Rdb\mapsto [T(1+is)](x) \in S^1(H_m)\] are continuous.
    \end{enumerate}
Now, define \[M_0 := \sup_{s\in\mathbb R}\|T(is) : S^2(H_N)\to S^2(H_m)\|,\text{ and } M_1:= \sup_{s\in\mathbb R}\|T(1+is) : S^1(H_N)\to S^1(H_m)\|.\] We show that \(M_0\leq 1,\) and \(M_1\leq 1.\)
Let $s\in\Rdb$ and $x\in B(H_N)$. Since $P$ is contractive, we have \begin{equation*}
\begin{split}
    \|[T(is)](x)\|_2 &= \|h_m^{\frac{is}{2}}P(g_N^{\frac{is}{2}} x g_N^{\frac{is}{2}})h_m^{\frac{is}{2}}\|_2\\ &\leq \|h_m^{\frac{is}{2}}\|_\infty^2 \|P(g_N^{\frac{is}{2}} x g_N^{\frac{is}{2}})\|_2\\ &\leq \|g_N^{\frac{is}{2}} x g_N^{\frac{is}{2}}\|_2 \leq \|g_N^{\frac{is}{2}}\|_\infty^2 \|x\|_2 \leq \|x\|_2.
\end{split}
\end{equation*}
Since $x = p_Nxp_N$, and $p_N= h^{\frac{1}{2}}g_N^{\frac{1}{2}} =  g_N^{\frac{1}{2}}h^{\frac{1}{2}}$, using the definition of the map $V_1$, we have \begin{equation*}
    \begin{split}
        h_m^{\frac{1}{2}}P(g_N^{\frac{1+is}{2}} x g_N^{\frac{1+is}{2}})h_m^{\frac{1}{2}} &= p_m h^{\frac{1}{2}}P(h^{\frac{1}{2}}g_N^{1+\frac{is}{2}}xg_N^{1+\frac{is}{2}}h^{\frac{1}{2}})h^{\frac{1}{2}}p_m\\ &= p_mV_1(hg_N^{1+\frac{is}{2}}xg_N^{1+\frac{is}{2}}h)p_m\\ &= p_mV_1(g_N^{\frac{is}{2}}xg_N^{\frac{is}{2}})p_m.
    \end{split}
\end{equation*}
Then, using the fact that $V_1$ is a contraction on $S^1(H)$, we obtain  \begin{equation*}
\begin{split}
    \|[T(1+is)](x)\|_1 &= \|h_m^{\frac{1+is}{2}}P(g_N^{\frac{1+is}{2}} x g_N^{\frac{1+is}{2}})h_m^{\frac{1+is}{2}}\|_1\\ &\leq \|h_m^{\frac{is}{2}}\|_\infty^2 \|h_m^{\frac{1}{2}}P(g_N^{\frac{1+is}{2}} x g_N^{\frac{1+is}{2}})h_m^{\frac{1}{2}}\|_1\\  &\leq \|p_mV_1(g_N^{\frac{is}{2}}xg_N^{\frac{is}{2}})p_m\|_1\\ &\leq \|g_N^{\frac{is}{2}}xg_N^{\frac{is}{2}}\|_1\\ &\leq  \|g_N^{\frac{is}{2}}\|_\infty^2 \|x\|_1 \leq \|x\|_1.
\end{split}
\end{equation*}
Consequently, by Stein's interpolation theorem, we obtain that for every $1\leq p\leq 2$ with $\frac{1}{p} = \frac{1-\theta}{2} + \theta$, the map $T(\theta) = T_{p,m,N} : S^p(H_N)\to S^p(H_m)$ is a contraction. 
So for every $N,m\in \Ndb$, every $x\in p_NB(H)p_N$, \[\|p_mV_p(x)p_m\|_p \leq \|x\|_p.\]
Since $p_mV_p(x)p_m \xrightarrow[m\to\infty]{S^p}V_p(x)$, we deduce that $V_p$ is a contraction on $\bigcup\limits_{N\in\mathbb N} p_NB(H)p_N$ and hence extends to a contraction on $S^p(H)$. 

Moreover, using similar arguments than for $V_1$, we obtain that $V_p$ is a positive projection.

\end{proof}

\begin{remark}
The idea behind the construction of the map $V_p$ is inspired by \cite[Theorem 5.1 and Section 7]{AR}, where some contractive projections $\mathbb E$ and $\mathbb{E}_p$ are introduced. The maps $V_\infty$ and $V_p$ defined in this paper can be seen as an analogue of $\mathbb E$ and $\mathbb{E}_p$ respectively, although the construction is simpler in the setting of Schatten spaces.
\end{remark}

For $2\leq q\leq \infty$, define the positive and contractive projection $V_q$ as follows: \[V_q:= V_p^* : S^q(H)\to S^q(H),\text{ with }1\leq p\leq 2,~~\frac{1}{p}+ \frac{1}{q}=1.\] The rest of this subsection aims to link the subspaces $\Ran(V_p)$ and $\Ran(V_q)$, for $1<p<\infty$, $\frac{1}{p}+\frac{1}{q}=1$. The identities (\ref{defVp}) and (\ref{defVq}) enable us to obtain information on $X=\Ran(P)$.

Denote \(\alpha:= \frac{1}{2}-\frac{1}{q} = \frac{1}{p} - \frac{1}{2}\geq 0\). It is straightforward to check that $V_q$ is the unique bounded map on $S^q(H)$ satisfying the following identity: for every $x\in B(H)$,
\begin{equation}\label{defVq}
     h^\alpha V_q(h^{\frac{1}{q}}xh^\frac{1}{q})h^\alpha = P(h^\frac{1}{2}xh^\frac{1}{2}).
\end{equation}
The link between a contractive projection on Schatten spaces and its adjoint has been studied in \cite[Section 1 and Section 2]{AF3}. We recall here some important results from there.

\begin{definition}
Let $1<p,q<\infty$ with $\frac{1}{p}+\frac{1}{q}=1$. We define the map $N_p : S^p(H)\to S^q(H)$ by \(N_p(0)=0\), and for any non-zero operator with polar decomposition $x = u|x| \in S^p(H)$, \[N_p(x) = \dfrac{|x|^{p-1}u^*}{\|x\|^{p-2}_p}.\]
\end{definition}

For $x\in S^p(H)\setminus\{0\}$, $N_p(x)$ is the unique element of $S^q(H)$ for which \[\|N_p(x)\|_q = \|x\|_p,\quad \text{ and}\quad \langle x, N_p(x)\rangle = \|x\|^2_p.\] The map $N_p :S^p(H)\to S^q(H)$ is a bijection, with $N_q$ serving as its inverse.
Note that for $x\in S^p(H)$, we have \(\text{s}_\ell(x) = \text{s}_r(N_p(x))\), and \(\text{s}_r(x) = \text{s}_\ell(N_p(x))\). Combining with \cite[Proposition 1.1, p.8 and Proposition 2.1, p.18]{AF3},  we obtain the following proposition.

\begin{proposition}\label{link_P_P*}
Let $1<p,q<\infty$ with $\frac{1}{p}+\frac{1}{q}=1$, and $Q : S^p(H)\to S^p(H)$ a contractive projection.
Then $Q^* : S^q(H)\to S^q(H)$ is a contractive projection, and we have 
\begin{enumerate}
    \item \(\text{Ran}(Q) = N_p(\text{Ran}(Q^*))\).
    \item $\Ran(Q)$ is indecomposable if, and only if, $\Ran(Q^*)$ is.
    \item \(\operatorname{s}_\ell(Q) = \operatorname{s}_r(Q^*),\) and \(\operatorname{s}_r(Q) = \operatorname{s}_\ell(Q^*)\).
\end{enumerate}
\end{proposition}
In the sequel, we fix $1<p<2$, and $2<q<\infty$ with $\dfrac{1}{p}+\dfrac{1}{q}=1$. We denote \(X_p:=\Ran(V_p)\), \(X_q:= \Ran(V_q)\). Then using identities (\ref{defVp}) and (\ref{defVq}), we obtain \begin{equation}\label{link_Vp_X_Vq}
    X_p = \overline{h^\alpha X h^\alpha}^{\|\cdot\|_p},\quad\text{and}\quad X = \overline{h^\alpha X_q h^\alpha}^{\|\cdot\|_2}.
\end{equation}

\begin{proposition}\label{XpXq}
With the previous notations, we have \begin{itemize}
    \item $h^{\frac{2}{p}}\in X_p$, and $h^{\frac{2}{q}}\in X_q.$
    \item The subspaces $X_p$ and $X_q$ are non-degenerate.
    \item The subspaces $X_p$ and $X_q$ are indecomposable.
\end{itemize}
\end{proposition}
\begin{proof}
The first point follows from the identities (\ref{defVp}), (\ref{defVq}) applied with $x=1$, and the second point from the first one. 
According to Proposition \ref{link_P_P*} (2), $X_p$ is indecomposable if and only if $X_q$ is indecomposable. So for the last point, it suffices to prove that $X_q$ is indecomposable.
Assume that $X_q$ is decomposable, that is, there are two non-trivial operator-disjoint subspaces $Y,Z$ of the space $S^q(H)$ such that \(X_q = Y\oplus Z.\) Let us show that it implies that the space $X$ is decomposable, which is absurd. Since $h^{\frac{2}{q}}\in X_q$ with $\text{s}(h^{\frac{2}{q}}) = \text{s}(X_q) =1$, there is $(y,z)\in Y\times Z$, such that $h^{\frac{2}{q}}=y+z$, and $y,z\neq 0$. Using the positivity of $h^{\frac{2}{q}}$, and the relations \(yz^*=y^*z=0\), we obtain \[h^{\frac{2}{q}} = |y|+|z| = y +z.\] We deduce that $y$ and $z$ are positive, and using continuous functional calculus, we have \[h^\alpha = y^\beta + z^\beta,\quad \text{with }\beta = \frac{\alpha q}{2}>0.\]
Furthermore, since $y$ is disjoint from every element of $Z$, we obtain \(y^nZ = \{0\}\) for every positive integer $n$, so we deduce \[y^\beta Z = \{0\}.\]
Similarly, we have \[z^\beta Y = \{0\}.\]
Hence \[h^\alpha X_q h^\alpha = y^\beta Y y^\beta + z^\beta Z z^\beta,\] and $y^\beta Y y^\beta$ and $z^\beta Z z^\beta$ are two operator-disjoint subspaces of $S^2(H)$.
Then, with relations (\ref{link_Vp_X_Vq}) and (\ref{relation_norm_disjoint}), we obtain \[ X = \overline{y^\beta Y y^\beta}^{\|.\|_2} \overset{2}{\oplus} \overline{z^\beta Z z^\beta}^{\|.\|_2}.\]
We deduce that $X$ is decomposable. A contradiction arises, so $X_q$ must be indecomposable.

\end{proof}


\subsection{Description of the positively 1-complemented subspace $X\subset S^2(H)$} We continue to use the notation defined in the previous subsection. Recall that $X$ is a non-degenerate indecomposable positively 1-complemented subspace of $S^2(H)$. We obtain a description of $X$ by applying the case $q\neq 2$ of Theorem \ref{THM1} to $X_q$, and then using the relationship between $X$ and $X_q$.


We begin with a proposition concerning $JC^*$-triples, which is useful for relating the types of $X_p$ and $X_q$. If $J\subset B(H)$ denotes a $JC^*$-triple, then for $1\leq p<\infty$, set $J_p = J\cap S^p(H)$. Note that if $J$ is weak*-closed, then $J_p\subset S^p(H)$ is closed.

\begin{proposition}\label{Np(JC*)}
Let $J\subset B(H)$ be a $JC^*$-triple, and let $1<p,q<\infty$, with $\frac{1}{p}+\frac{1}{q}=1$.

We have \[N_p(J_p) = \left\{x^*,\quad x\in J_q\right\}.\]
\end{proposition}

\begin{proof}
If $x\in J_p$ is a non-zero operator with polar decomposition \(x=u|x|\), then by Proposition \ref{propfunctiononJC}, $u|x|^{p-1}\in J$. We deduce the inclusion \[N_p(J_p)\subset\{x^*,\quad x\in J_q\}.\]

The other inclusion follows from the fact that the map $N_p$ is adjoint preserving and bijective, with $N_q$ as its inverse.
\end{proof}

Recall that $X$ is a non-degenerate, positively 1-complemented, and indecomposable subspace of $S^2(H)$.  We fix $1<p<2$ and $2<q<\infty$ such that $\frac{1}{p}+\frac{1}{q}=1$. Then by Proposition \ref{XpXq}, $X_p$ and $X_q$ are non-degenerate, indecomposable, positively 1-complemented subspaces of $S^p(H)$ and $S^q(H)$ respectively.

We will assume that $X_q$ is of one of the types described in Theorem  \ref{THM1}, and examine the implications for $X_p$ and $X$. The following proposition handles the types of symmetric matrices, anti-symmetric matrices and odd dimensional spinorial spaces.

\begin{proposition}\label{fincas126}
Let $K$ be a Hilbert space, and $J\subset B(K)$ be a $JC^*$-triple invariant under $*$ (i.e. for every $t\in J$, $t^*\in J$), that is weak*-closed in $B(H)$. Let $a\in S^q(H_1)$ be a positive injective operator on some Hilbert space $H_1$, with $\|a\|_q=1$ and let $V : K\overset{2}{\otimes}H_1\to H$ be a unitary operator.
Assume that the subspace $X_q$ may be described as follows: \[X_q = \{V(t\otimes a)V^*,\quad t\in J_q\}.\]

Then \[ X_p = \{V(y\otimes a^{q-1})V^*,\quad y\in J_p\}, \quad \text{and}\quad X = \{V(z\otimes a^{\frac{q}{2}})V^*,\quad z\in J_2\},\]
\end{proposition}
\begin{proof}
For the first equality, we only need to compute $N_q(X_q)$. We may always assume that $\|a\|_q=1$.

Let $t\in J_q$, and set $x:= V(t\otimes a)V^* \in X_q$. We have \[|x| = V(|t|\otimes a)V^*,\quad \text{and}\quad \|x\|_q=\|t\|_q\|a\|_q = \|t\|_q.\] If $t = u|t|$ is the polar decomposition of $t$, then the polar decomposition of $x$ is \[x= V(u\otimes 1)V^*|x|,\] and we obtain \[N_q(x) = \dfrac{|x|^{q-1}(V(u\otimes 1)V^*)^*}{\|t\|_q^{q-2}} = \dfrac{V(|t|^{q-1}\otimes a^{q-1})(u^*\otimes 1)V^*}{\|t\|_q^{q-2}} = V(N_q(t)\otimes a^{q-1})V^*.\]

Hence \[X_p = \{V(N_q(t)\otimes a^{q-1})V^*,\quad t\in J_q\} = \{V(y\otimes a^{q-1})V^*,\quad y\in J_p\},\] by Propositions \ref{link_P_P*} and \ref{Np(JC*)}.

\medskip

For the second equality, we use (\ref{link_Vp_X_Vq}), applying the two identities separately for each inclusion. First, we may compute the space \(h^\alpha X_q h^\alpha\), with $\alpha = \frac{1}{2}- \frac{1}{q}$. Note that $h^{2/q}\in X_q$, so there exists a positive element $t_0\in J_q$, such that \( h^{2/q} = V(t_0\otimes a)V^*.\) We deduce that \(h^\alpha = V(t_0^{\frac{q\alpha}{2}}\otimes a^{\frac{q\alpha}{2}})V^*,\) and we have \[ h^\alpha X_q h^\alpha = \{V(t_0^{\frac{q\alpha}{2}}yt_0^{\frac{q\alpha}{2}}\otimes a^{1+q\alpha})V^*,\quad y\in J_q\}.\]
Note that \(1+q\alpha = \frac{q}{2}\). Since $t_0^\frac{q\alpha}{2} \in S^{2/\alpha}(H)$, and $\dfrac{1}{2} = \dfrac{\alpha}{2} + \dfrac{1}{q} + \dfrac{\alpha}{2}$, we deduce by Hölder inequality that $t_0^{\frac{q\alpha}{2}}yt^{\frac{q\alpha}{2}} \in S^2(H)$.

In addition, $t_0$ is a positive element of $J$, so by Proposition \ref{propfunctiononJC}, we deduce that $ t_0^{\frac{q\alpha}{2}}\in J$. Then, for every $y\in J_q$ we have $y^*\in J_q$, so the triple product $\{t_0^{\frac{q\alpha}{2}},y^*,t_0^{\frac{q\alpha}{2}}\} = t_0^{\frac{q\alpha}{2}}yt_0^{\frac{q\alpha}{2}}$ remains in $J$.
We deduce the following inclusion \begin{equation}\label{isomJ_2}
    h^\alpha X_q h^\alpha \subset \{V(z\otimes a^{q/2})V^*,\quad z\in J_2\}.
\end{equation}
Note that $J\subset B(H)$ is weak*-closed, hence the subspace $J_2$ of $S^2(H)$ is closed. The space in the right-hand side of (\ref{isomJ_2}) is isometric to $J_2\subset S^2(H)$, so it is closed subspace of $S^2(H)$, and we have \[X = \overline{h^\alpha X_q h^\alpha}^{\|\cdot\|_2} \subset \{V(z\otimes a^{q/2})V^*,\quad z\in J_2\}.\]

\medskip

Conversely, let $z\in J_2$, and consider $x:= V(z\otimes a^{q/2})V^*$. Let us show that \(P(x)=x.\)
For the same reasons as above, we have \(h^\alpha x h^\alpha \in X_p\), and we deduce that \[V_p(h^\alpha x h^\alpha) = h^\alpha x h^\alpha.\]
As in the previous subsection, we denote $p_N$ the orthogonal projection onto $\text{span}\{e_1,\ldots ,e_N\},$ where $(e_N)_N$ is an orthonormal basis resulting from the diagonalization of $h$. Then, we both have \[V_p(p_Nh^\alpha x h^\alpha p_N)\xrightarrow[N\to +\infty]{\|\cdot\|_p}V_p(h^\alpha x h^\alpha) = h^\alpha x h^\alpha,\] and \[V_p(p_Nh^\alpha x h^\alpha p_N) = h^\alpha P(p_Nxp_N)h^\alpha \xrightarrow[N\to +\infty]{\|\cdot\|_p}h^\alpha P(x) h^\alpha.\]
Since $h^\alpha$ is an injective operator with dense range, we deduce the equality \[P(x) =x.\]

\end{proof}


We continue with a proposition that settles the two remaining types.

\begin{proposition}\label{fincas35}
Let $J, J'\subset B(K)$ be unital weak*-closed $JC^*$-triples on a Hilbert space $K$, $\psi : J\to J'$ be a unital triple-isomorphism which restricts to an isometry $\psi: J_p\to J_p'$, let $a_1\in S^p(H_1)$, $a_2\in S^q(H_2)$ be two positive injective operators on Hilbert spaces, and assume that we have an orthogonal decomposition \(H=H_1'\overset{\perp}{\oplus}H_2'\), with two unitaries \(V_i: K\overset{2}{\otimes}H_i\to H_i'\subset H\), for $i=1,2$ such that the space $X_q$ may be described as \[X_q = \left\{V_1(t\otimes a_1)V_1^* + V_2(\psi(t)\otimes a_2)V_2^*,\quad t\in J_q\right\}.\]
Then we have \[X_p = \left\{V_1(y\otimes a_1^{q-1})V_1^* + V_2(\psi(y)\otimes a_2^{q-1})V_2^*,\quad y\in J_p\right\},\] and \[X = \left\{V_1(z\otimes a_1^{q/2})V_1^* + V_2(\psi(z)\otimes a_2^{q/2})V_2^*,\quad z\in J_2\right\}.\]
\end{proposition}
\begin{proof}
First, we determine $N_q(X_q)$. We may always assume $\|a_1\|_q^q + \|a_2\|_q^q = 1$. The computations are very similar to those in the previous proposition.
For $t\in J_q$, let $x = V_1(t\otimes a_1)V_1^* + V_2(\psi(t)\otimes a_2)V_2^*\in X_q$. Then we have \[|x| = V_1(|t|\otimes a_1)V_1^* + V_2(|\psi(t)|\otimes a_2)V_2^*,\] and \begin{equation*}
\begin{split}
    \|x\|_q^q &= \|V_1(t\otimes a_1)V_1^*\|_q^q + \|V_2(\psi(t)\otimes a_2)V_2^*\|_q^q \\ &= \|t\otimes a_1\|_q^q + \|\psi(t)\otimes a_2\|_q^q\\ &= \|t\|_q^q\|a_1\|_q^q + \|\psi(t)\|_q^q\|a_2\|_q^q = \|t\|_q^q,
\end{split}
\end{equation*} using the identity $\|a_1\|_q^q + \|a_2\|_q^q = 1$ and the fact that $\psi$ is an isometry. Denote by \(t=u|t|\) and \(\psi(t) = w|\psi(t)|\) their polar decompositions, then the polar decomposition of $x$ is given by \[x = (V_1(u\otimes 1)V_1^* + V_2(w\otimes 1)V_2^*)|x|.\]
We can compute 
\begin{equation*}
\begin{split}
        N_q(x) &= \dfrac{|x|^{q-1}(V_1(u\otimes 1)V_1^* + V_2(w\otimes 1)V_2^*)^*}{\|x\|_q^{q-2}}\\ &= \dfrac{(V_1(|t|^{q-1}\otimes a_1)V_1^* + V_2(|\psi(t)|^{q-1}\otimes a_2)V_2^*)(V_1(u\otimes 1)V_1^* + V_2(w\otimes 1)V_2^*)^*}{\|t\|_q^{q-2}}\\ &= \dfrac{V_1(|t|^{q-1}u^*\otimes a_1)V_1^* + V_2(|\psi(t)|^{q-1}w^*\otimes a_2)V_2^*}{\|t\|_q^{q-2}}\\  &= V_1(N_q(t)\otimes a_1)V_1^* + V_2(N_q(\psi(t))\otimes a_2)V_2^*.
\end{split}
\end{equation*}
Applying Proposition \ref{propfunctiononJC} with the function $f : s\in\Rdb_+\mapsto s^{q-1}$ and by Proposition \ref{Np(JC*)}, we obtain \begin{equation*}
    \begin{split}
        N_q(\psi(t)) &= \dfrac{|\psi(t)|^{q-1}w^*}{\|\psi(t)\|_q^{q-2}}\\ &= \dfrac{1}{\|t\|_q^{q-2}}(wf(|\psi(t)|))^*\\ &= \dfrac{1}{\|t\|_q^{q-2}} \psi(uf(|t|))^*\\ &= \dfrac{1}{\|t\|_q^{q-2}} \psi(f(|t|)u^*)\\ &= \psi\left(\dfrac{1}{\|t\|_q^{q-2}} |t|^{q-1}u^* \right) = \psi(N_q(t)).
    \end{split}
\end{equation*}
We also used the fact that $N_q$ is adjoint preserving, and the equality \(\|\psi(t)\|_q = \|t\|_q.\) We deduce \begin{equation*}
    \begin{split}
       X_p =  N_q(X_q) &= \{V_1(N_q(t)\otimes a_1)V_1^* + V_2(\psi(N_q(t))\otimes a_2)V_2^*,\quad t\in J_q\}\\ &= \{V_1(y\otimes a_1)V_1^* + V_2(\psi(y)\otimes a_2)V_2^*,\quad y\in J_p\}.
    \end{split}
\end{equation*}

 The second equality may also be proven in a similar way as in the previous proposition. For the first inclusion $\subset$, we need to write the operator $h^\alpha$ as: \[h^\alpha = V_1(t_0^{\frac{\alpha q}{2}}\otimes a_1^{\frac{\alpha q}{2}})V_1^* + V_2(\psi(t_0^{\frac{\alpha q}{2}})\otimes a_2^{\frac{\alpha q}{2}})V_2^*,\] where $t_0$ is a positive element of $J_q$. Note that $J$ and $J'$ are operator systems, and since $\psi : J\to J'$ is a unital contraction, it is a positive map (see \cite[Proposition 2.11, p.16]{Paulsen}). We deduce, using Proposition \ref{propfunctiononJC}, that \(\psi(t_0)^{\frac{\alpha q}{2}} = \psi(t_0^\frac{\alpha q}{2})\) and for all $y\in J_q$, we have $\psi(t_0^{\frac{\alpha q}{2}})\psi(y)\psi(t_0^{\frac{\alpha q}{2}}) = \psi(t_0^{\frac{\alpha q}{2}}yt_0^{\frac{\alpha q}{2}})$. We can show that \[ X\subset \left\{V_1(z\otimes a_1^{q/2})V_1^* + V_2(\psi(z)\otimes a_2^{q/2})V_2^*,\quad z\in J_2\right\}.\]
The converse inclusion $\supset$ is established in the same manner as in Proposition \ref{fincas126}.
\end{proof}

Now, we can prove Theorem \ref{THM1} for $p=2$.
\begin{proof}
Consider $X$ an indecomposable positively 1-complemented subspace of the space $S^2(H)$. As explained just after Theorem \ref{THM1}, we may assume that $X\subset S^2(H)$ is non-degenerate. Then, with the construction of Subsection 5.1, for $1<p<2$ and $2<q<\infty$ such that $\frac{1}{p} + \frac{1}{q} =1$, we obtain two non-degenerate indecomposable positively 1-complemented subspaces \[X_p\subset S^p(H),~X_q\subset S^q(H).\]

Since $1<q\neq 2<\infty$, we can apply Theorem \ref{THM1} to $X_q$. Assume that we are in one of the following case: There exist $a\in S^q(H_1)$ positive injective, and \begin{enumerate}
    \item there exist a countable index set $I$, a unitary operator $O_1\in\S_I$, such that \[X_q\simp O_1\S_I^q\otimes a.\]
    \item There exist a countable index set $I$, a unitary operator $O_2\in\A_I$, such that \[X_q\simp O_2\A_I^q\otimes a.\]
    \item There exist an integer $N\geq 2$, a unitary operator $v\in \E_{2N}$, such that \[X_q\simp v\E_{2N}^q\otimes a.\]
\end{enumerate}
Note that the spaces $O_1\S_I$, $O_2\A_I$ and $v\E_{2N}$ are unital $JC^*$-triples. By Lemma \ref{JCautoadj}, these spaces are invariant under $*$, and we can apply Proposition \ref{fincas126} to them. So, we are in one of these three situations.
\begin{enumerate}
    \item \(X\simp O_1\S_I^2\otimes a^{\frac{q}{2}}.\)
    \item \(X\simp O_2\A_I^2\otimes a^{\frac{q}{2}}.\)
    \item \(X\simp v\E_{2N}^2\otimes a^{\frac{q}{2}}\).
\end{enumerate}
Note that $a^{\frac{q}{2}}$ is a positive injective operator in $S^2(H)$.

Likewise, assume there exist $a_1\in S^q(H_1)$, $a_2\in S^q(H_2)$ two positive injective operators such that we are in one of the following two cases:
\begin{enumerate}[start=4]
    \item there exists a countable index set $I$, such that \[X_q\simp \{(w\otimes a_1, w^\top\otimes a_2),\quad w\in S^q_I\}.\]
    \item There exist $N\geq 2$, and a unitary operator $v\in \F_N$, such that \[X_q \simp \{(vx\otimes a_1, \sigma(v)\sigma(x)\otimes a_2),\quad x\in \F_N^q\} = \{(y\otimes a_1, \sigma(v)\sigma(v^*y)\otimes a_2),\quad y\in v\F_N^q\}.\]
\end{enumerate}
Note that $B(\ell^2_I)$, $v\F_N$ and $\sigma(v)\F_N$ are unital $JC^*$-triples. In addition, the maps $\top : B(\ell^2_I)\to B(\ell^2_I)$ and $\psi :v\F_N\to \sigma(v)\F_N$ defined by \(\psi(y) = \sigma(v)\sigma(v^*y)\) are unital triple-isomorphism and isometries on $S^p_I$ and $v\F_N^p$ respectively. So we can apply Proposition \ref{fincas35}, and we are in one of these situations: 
\begin{enumerate}[start=4]
    \item \(X\simp \{(w\otimes a_1^{\frac{q}{2}}, w^\top\otimes a_2^{\frac{q}{2}}),\quad w\in S^2_I\}\)
    \item \(X\simp \{(vx\otimes a_1^{\frac{q}{2}}, \sigma(v)\sigma(x)\otimes a_2^{\frac{q}{2}}),\quad y\in \F_N^2\},\)
\end{enumerate} with $a_1^{\frac{q}{2}}$, $a_2^{\frac{q}{2}}$ two positive injective operators.

\end{proof}

\section{Appendix: Equivalence in Types 5 and 6}\label{sectionAppendix}
The spaces called spinorial spaces (types 5 and 6) in Theorem \ref{thmAF} are not originally defined in the same way in the memoirs of Arazy and Friedman (see \cite[Section 2.c and 2.d]{AF2} and \cite{AF3}). The purpose of this section is to show that the definition of elementary subspaces of type 5 (respectively type 6) given by Arazy-Friedman is equivalent to the definition of subspaces of type 5 (resp. type 6) adopted in the present article. We begin by defining the type 5 and 6 in the sense of Arazy-Friedman.

We use the same notations as in Section 2, which are used to define Types 4, 5 and 6. Let $N\geq 2$ and let $\La_N^e$ (resp. $\La_N^o$) be the subspace of $\La_N$ spanned by even tensors (resp. odd tensors). In particular, we have \(\La_N = \La_N^e \overset{\perp}{\oplus}\La_N^o.\) For $1\leq k\leq N$, recall that $c_k : \La_N\to \La_N$ denotes the creation operator, defined by $x\mapsto e_k\wedge x$. Let \(x_k : \La_N^e \to \La_N^o\) and \(\Tilde{x}_k : \La_N^e\to \La_N^o\) be the restrictions of $c_k$ and $c_k^*$ to \(\La_N^e\) respectively.
Define \[AH(N) := \text{span}\{x_k,\Tilde{x}_k,\quad 1\leq k\leq N\},\quad BH(N):= \{x^*,\quad x\in AH(N)\}\] and \[ DAH(N) := \text{span}\{x_j,\tilde{x}_j, x_1+\tilde{x}_1,~2\leq j\leq N\}.\] We write $AH(N)^p$ (respectively $BH(N)^p$ and $DAH(N)^p$) to denote the space $AH(N)$ (resp. $BH(N)$ and $DAH(N)$) endowed with the norm of $S^p(\La_N^e, \La_N^o)$ (resp. $S^p(\La_N^o, \La_N^e)$ and $S^p(\La_N^e, \La_N^o)$). Let $T : AH(N)^p\to BH(N)^p$ be the linear map defined by \[T(x_k) = \Tilde{x}_k^*\quad \text{ and }\quad T(\Tilde{x}_k) = x_k^*\quad 1\leq k\leq N.\] This map is a surjective isometry; see \cite[Section 2]{MRR} for more details.

\begin{definition}\label{defAHN AF}
Let $1\leq p<\infty$ and let $H$ be a Hilbert space.
\begin{enumerate}
    \item Let $N\geq 2$. A subspace $X$ of the space $S^p(H)$  is called elementary subspace of type 5 (in the sense of Arazy-Friedman) when $X$  is equivalent to the subspace $$\{(x\otimes a_1, T(x)\otimes a_2),\quad x\in AH(N)^p\},$$ for some operators $a_1\in S^p(H_1)$ and $a_2\in S^p(H_2)$.
    \item Let $N\geq 3$. A subspace $X$ of the space $S^p(H)$  is called elementary subspace of type 6 (in the sense of Arazy-Friedman) when $X$  is equivalent to the subspace $$\{x\otimes a,\quad x\in DAH(N)^p\},$$ for some operator $a\in S^p(H_1)$.
\end{enumerate}
\end{definition}

The following proposition immediately implies that the subspaces of type 5 (resp. type 6) in the sense of Arazy-Friedman coincide with the subspaces of type 5 (resp. type 6) in the sense of the present article.

\begin{proposition}\label{propequivFnAH}Let $1\leq p<\infty$ and let $N\geq 2$. Let $a,a_1\in S^p(H_1)$, let $a_2\in S^p(H_2)$.
\begin{itemize}
    \item The space \(\{(x\otimes a_1, T(x)\otimes a_2),\quad x\in AH(N+1)^p\}\) is equivalent to the spinorial space of even dimension \(\{(x\otimes a_1, \sigma(x)\otimes a_2),\quad x\in\F_N^p\}\).
    \item The space \(\{x\otimes a,\quad x\in DAH(N+1)^p\}\) is equivalent to the spinorial space of odd dimension \(\{x\otimes a,\quad x\in\E_{2N}^p\}\).
\end{itemize}
\end{proposition}
\begin{remark}
The spaces described in Definition \ref{defAHN AF} are precisely those appearing in the classification of indecomposable 1-complemented subspaces in \cite{AF2} and \cite{AF3}. In the present work, we used a description with the spaces $\F_N$ and $\E_{2N}$, as it is more convenient for the study of the positivity of the contractive projection onto these spaces.
In \cite{MRR}, a link between $AH(N+1)$, $BH(N+1)$ and $\F_N$ is given (as well as between $DAH(N+1)$ and $\E_{2N}$). We briefly review it here, with a slight extension provided by the following lemma.
\end{remark}

We use the following elementary lemma to prove Proposition \ref{propequivFnAH}, and we rely on \cite[Section 5]{MRR}.
 \begin{lemma}\label{equiv_*isom}
Let $K$ be a finite-dimensional Hilbert space, and let $H$ be a Hilbert space. Let $\rho : B(K)\to B(H)$ be an injective $*-$homomorphism that preserves the trace. Then there exists an isometry $U : K\to H$ such that \[\rho(x) = UxU^*,\quad x\in B(K).\]
\end{lemma}

\begin{proof}[Proof of Proposition \ref{propequivFnAH}]
For $1\leq k\leq n$, consider \(w_k = c_k + c_k^* \in B(\La_n)\), and let \(\mathcal{C}_n = C^*\langle w_1, \ldots, w_n\rangle \subset B(\La_n)\) be the $C^*$ algebra generated by the family $(w_1,\ldots, w_n)$. 

Now, let $n=2N$ be an even integer. Define the following two subspaces of $\mathcal{C}_{2N} \subset B(\La_{2N})$: \[E_{2N} := \text{span}\{1,w_1,\ldots, w_{2N}\}, \text{ and } F_N := \text{span}\{1,w_1,\ldots, w_{2N}, w_1\cdots w_{2N}\}.\] Let $\Tilde{\sigma} : F_N\to F_N$ be the linear map defined by \(\Tilde{\sigma}(1)=1,\)  \(\Tilde{\sigma}(w_k) = w_k, \text{ for every }1\leq k\leq 2N,\) and \(\tilde{\sigma}(w_1\cdots w_{2N}) = -w_1\cdots w_{2N}.\) Let $v \in B(\La_{N+1})$ be the unitary operator defined by \[v := c_{N+1} + c_{N+1}^*.\]
According to Section 5 and to the proof of \cite[Theorem 5.11, p.871]{MRR}, there exist injective $*-$homomorphisms \(\pi_0 : \mathcal{C}_{2N}\to \mathcal{C}_{2N+1}\), \(\pi_1 : \mathcal{C}_{2N}\to \mathcal{C}_{2N+1}\), and an injective unital $*$-homomorphism $\pi : \mathcal{C}_{2N+1}\to B(\La_{N+1})$ such that \begin{equation}\label{lienFnAHn}
    \pi\pi_0(F_N) = ivAH(N+1),\quad\pi\pi_0(E_{2N}) = ivDAH(N+1), \text{ and } \pi\pi_1(F_N) = ivBH(N+1).
\end{equation} Moreover, we have \begin{equation}\label{transpositions}
    \pi\pi_1\tilde{\sigma}(y) = iv T((iv)^*\pi\pi_0(y)),\quad y\in F_N.
\end{equation}
Using the elements $s_j$, $\F_N$ and $\E_{2N}$ defined in (\ref{defs_j}), (\ref{defFn}) and (\ref{defE2N}), we apply \cite[Lemma 5.9, p.869]{MRR}. There exists an $*-$isomorphism $\Phi : B(\La_N)\to \mathcal{C}_{2N}$ such that \[\Phi(s_j) = w_{2N-2(j-1)}\quad \text{and}\quad \Phi(s_{-j})= w_{2N+1-2j},\quad 1\leq j\leq N.\] Then \[\Phi(\F_N) = F_N,\quad \Phi(\E_{2N}) = E_{2N}, \quad \text{and }\quad \sigma = \Phi^{-1}\circ\tilde{\sigma}\circ\Phi.\]
Note that $\Phi$ and $\pi$ are unital, and that $\pi_0$ and $\pi_1$ are trace preserving as explained just after \cite[(5.6), p.867]{MRR}, noting that the trace $\Tr$ of the article \cite{MRR} is the normalised trace. We obtain that the $*-$homomorphisms $\pi\pi_0\Phi$ and $\pi\pi_1\Phi$ preserve the trace. We apply Lemma \ref{equiv_*isom}: there exist isometries $W_0,W_1 : \La_N\to \La_{N+1}$ such that \[\pi\pi_0\Phi(x) = W_0xW_0^*,\quad \text{and}\quad \pi\pi_1\Phi(x) = W_1xW_1^*.\]
From this and from equalities (\ref{lienFnAHn}), we deduce \begin{equation}\label{equiv_FnAHn}
    ivAH(N+1) = W_0\F_N W_0^*;~~ ivBH(N+1)= W_1\F_NW_1^*;~~ ivDAH(N+1)= W_0\E_{2N}W_0^*.
\end{equation}
Using (\ref{transpositions}), we deduce that for every $x\in AH(N+1)$, there exists $y\in\F_N$ such that \[T(x) = (iv)^*W_1\sigma(y)W_1^*.\]
Let $x\in AH(N+1)$, and set $z = (x\otimes a_1, T(x)\otimes a_2)$. 
Then, there exists $y\in\F_N$ such that $ivx = W_0yW_0^*$.
We have \begin{equation*}
    \begin{split}
         z = (x\otimes a_1, T(x)\otimes a_2) &= ((iv)^*W_0yW_0^*\otimes a_1, T((iv)^*W_0yW_0^*)\otimes a_2)\\ &=  ((iv)^*W_0yW_0^*\otimes a_1, T((iv)^*\pi\pi_0\Phi(y))\otimes a_2)\\ &= ((iv)^*W_0yW_0^*\otimes a_1, (iv)^*\pi\pi_1(\Phi(\sigma(y)))\otimes a_2)\\ &= ((iv)^*W_0yW_0^*\otimes a_1, (iv)^*W_1\sigma(y)W_1^*\otimes a_2)\\ &= ((iv)^*W_0\otimes 1, (iv)^*W_1\otimes 1)(y\otimes a_1, \sigma(y)\otimes a_2)(W_0\otimes 1, W_1\otimes 1)^*.
    \end{split}
\end{equation*}
Conversely, for any $y\in\F_N$, there exists $x\in AH(N+1)$ such that \[(y\otimes a_1, \sigma(y)\otimes a_2) = ((iv)^*W_0\otimes 1, (iv)^*W_1\otimes 1)^*(x\otimes a_1, T(x)\otimes a_2)(W_0\otimes 1, W_1\otimes 1).\]
We deduce the following equivalence \[\{(x\otimes a_1, T(x)\otimes a_2),\quad x\in AH(N+1)^p\}\sim \{(x\otimes a_1, \sigma(x)\otimes a_2),\quad x\in\F_N^p\}.\]
For the second equivalence the computations are similar.
\end{proof}

\vspace{2cm}
\textbf{Acknowledgment.} The author would like to thank her thesis supervisor, Christian Le Merdy, for his valuable advice and support.

\end{document}